\documentclass[11pt]{amsart}

\usepackage[utf8]{inputenc}
\usepackage[T1]{fontenc}
\usepackage[english]{babel}

\numberwithin{equation}{section}

\usepackage{amsthm} 
\usepackage{amsmath}
\usepackage{amssymb} 
\usepackage{mathrsfs}
\usepackage{bbold}
\usepackage{graphicx}
\usepackage{listings}
\usepackage{mathrsfs}
\usepackage{enumerate}
\usepackage{pict2e}
\usepackage{stmaryrd}
\usepackage{mathtools}
\usepackage{extarrows}

\usepackage[all,cmtip]{xy}

\usepackage{color}
\definecolor{marin}{rgb}   {0.,   0.1,   0.5} 
\definecolor{rouge}{rgb}   {0.8,   0.,   0.} 
\definecolor{sepia}{rgb}   {0.4,   0.25,   0.} 
\definecolor{mag}{rgb}   {0.3,   0,   0.3} 
\usepackage[colorlinks,citecolor=sepia,linkcolor=marin,urlcolor=mag ,bookmarksopen,bookmarksnumbered]{hyperref}

\newtheorem{theorem}{Theorem}

\newtheorem{lemma}{Lemma}[section]
\newtheorem{proposition}[lemma]{Proposition}
\newtheorem{corollary}[lemma]{Corollary}

\theoremstyle{definition}
\newtheorem{definition}[lemma]{Definition}
\newtheorem{remark}[lemma]{Remark}

\newcommand{\Z}{\mathbb{Z}}

\newcommand{\ii }{{\rm i} }

\newcommand{\s}{\sigma}

\newcommand{\pa}{\partial}

\newcommand{\T}{\mathbb{T}}
\newcommand{\R}{\mathbb{R}}

\newcommand{\eps}{\varepsilon}

\providecommand{\vect}[2]{{\bigl[\begin{smallmatrix}#1\\#2\end{smallmatrix}\bigr]}}   
\providecommand{\sm}[4]{{\bigl[\begin{smallmatrix}#1&#2\\#3&#4\end{smallmatrix}\bigr]}}

\DeclareMathOperator{\tiret}{-}

\begin{document}

\title[Long-time existence for beam equations on irrational tori]{{L}ong-time 
existence for semi-linear beam equations on  irrational tori.}

\author{Joackim Bernier }
\address{\scriptsize{Laboratoire de Math\'ematiques Jean Leray, Universit\'e de Nantes, UMR CNRS 6629\\
2, rue de la Houssini\`ere \\
44322 Nantes Cedex 03, France}}
\email{joackim.bernier@univ-nantes.fr}

\author{Roberto Feola}
\address{\scriptsize{Laboratoire de Math\'ematiques Jean Leray, Universit\'e de Nantes, UMR CNRS 6629\\
2, rue de la Houssini\`ere \\
44322 Nantes Cedex 03, France}}
\email{roberto.feola@univ-nantes.fr}

 \author{ Beno\^it Gr\'ebert}
\address{\scriptsize{Laboratoire de Math\'ematiques Jean Leray, Universit\'e de Nantes, UMR CNRS 6629\\
2, rue de la Houssini\`ere \\
44322 Nantes Cedex 03, France}}
\email{benoit.grebert@univ-nantes.fr}

\author{Felice Iandoli}
\address{\scriptsize{Laboratoire Jacques-Louis Lions, Sorbonne Universit\'e, UMR CNRS 7598\\
4, Place Jussieu\\
 75005 Paris Cedex 05, France}}
\email{felice.iandoli@sorbonne-universite.fr}



\thanks{Felice Iandoli has been supported  by ERC grant ANADEL 757996. Roberto Feola, Joackim Bernier and Benoit Gr\'ebert have been supported by 
 the Centre Henri Lebesgue ANR-11-LABX- 0020-01 
 and by ANR-15-CE40-0001-02 ``BEKAM'' of the ANR}

\begin{abstract} 
We consider the semi-linear beam 
equation  on the $d$ dimensional irrational torus 
with smooth nonlinearity of order $n-1$ 
with $n\geq3$ and $d\geq2$. 
 If $\eps\ll 1$ is the size of the initial datum, 
 we prove that the lifespan $T_\eps$ of solutions   
 is $O(\eps^{-A(n-2)^-})$ where 
 $A\equiv A(d,n)= 1+\frac 3{d-1}$ when $n$ is even  and 
 $A= 1+\frac 3{d-1}+\max(\frac{4-d}{d-1},0)$ when $n$ is odd. 
 For instance for $d=2$ and $n=3$ (quadratic nonlinearity) 
 we obtain $T_\eps=O(\eps^{-6^-})$, much better than  $O(\eps^{-1})$, 
 the  time given by the local existence theory. 
 The irrationality of the torus makes the set of differences 
 between two eigenvalues of $\sqrt{\Delta^2+1}$ accumulate to zero, 
 facilitating the exchange between the high Fourier modes  
 and complicating the control of the solutions over long times. 
 Our result is obtained by combining a Birkhoff normal form step 
 and a modified energy step.
\end{abstract}

\maketitle

\setcounter{tocdepth}{1} 
\tableofcontents

\section{Introduction}
In this article we consider  the beam equation on an irrational torus
 \begin{equation}\label{eq:beam}
\left\{ \begin{aligned}
& \partial_{tt}\psi+\Delta^{2} \psi+\psi+f(\psi)=0 \,,\\
&\psi(0,y)=\psi_0\,,\\
&\partial_{t}\psi(0,y)=\psi_1\,,
 \end{aligned}\right.
 \end{equation}
 where $f\in C^{\infty}(\R,\R)$, $\psi=\psi(t,y)$,  
 $y\in \mathbb{T}^{d}_{{\nu}}$, with ${\nu}=(\nu_1,\ldots,\nu_{d})\in [1,2]^{d}$  
 and 
 \begin{equation}\label{toriIrr}
 \mathbb{T}^{d}_{{\nu}}:=(\mathbb{R}/ 2\pi \nu_1\mathbb{Z})\times 
 \cdots\times(\mathbb{R}/ 2\pi \nu_d \mathbb{Z})\,.
\end{equation}
The initial data $(\psi_0, \psi_1)$ have small size $\eps$ in the standard Sobolev space
 $H^{s+1}(\mathbb{T}_\nu^{d})\times H^{s-1}(\mathbb{T}_\nu^{d})$ for some $s\gg1$.
 The nonlinearity $f(\psi)$ has the form
 \begin{equation}\label{hypFF}
 f(\psi):=(\partial_{\psi}F)(\psi)
 \end{equation}
 for some 
  smooth function $F\in C^{\infty}(\R,\R)$ having 
 a  zero of order at least $n\geq3$ at the origin. 
 Local existence theory implies that \eqref{eq:beam}
 admits, for small $\eps>0$, a unique smooth solution 
 defined on an interval of length $O(\eps^{-n+2})$. 
 Our goal is to prove that, generically with respect 
 to the irrationality of the torus (i.e. generically with respect to the parameter $\nu$), 
 the solution actually extends to a larger interval.
 
 \noindent
Our main theorem  is the following.
\begin{theorem}\label{thm-main} Let $d\geq2$.
There exists $s_0\equiv s_0(n,d)\in\R$ such that for almost all 
$\nu\in[1,{2}]^{d}$, for any $\delta>0$
and for any $s\geq s_0$ there exists
 $\eps_0>0$ such that for any $0<\eps\leq \eps_0$
we have the following. 
For any initial data 
$(\psi_0,\psi_1)\in H^{s+1}(\mathbb{T}_{\nu}^{d})\times H^{s-1}(\mathbb{T}_{\nu}^{d})$
such that
\begin{equation}\label{initialstima}
\|\psi_0\|_{H^{s+1}}+\|\psi_1\|_{H^{s-1}}\leq \eps\,,
\end{equation}
there exists a unique solution 
of the Cauchy problem \eqref{eq:beam} such that 
\begin{equation}\label{tesiKG}
\begin{aligned}
&\psi(t,x)\in C^0\big([0,T_\eps);H^{s+1}(\T_\nu^d)\big)\bigcap 
C^1\big([0,T_\eps);H^{s-1}(\T_\nu^d)\big)\,, \\
& \sup_{t\in[0,T_\eps)}\Big(\|\psi(t,\cdot)\|_{H^{s+1}} +
\|\partial_t \psi(t,\cdot)\|_{H^{s-1}}    \Big)\leq 2\eps\,, 
\qquad T_\eps\geq  \eps^{-\mathtt{a}+\delta}\,,
\end{aligned}
\end{equation}
where $\mathtt{a}=\mathtt{a}(d,n)$ has the form
\begin{equation}\label{Atime}
\mathtt{a}(d,n):=
\left\{\begin{aligned}
&(n-2)\big(1+\tfrac{3}{d-1}\big)\,,\qquad\qquad\qquad\quad    
n\;\; {\rm even}\\
&(n-2)\big(1+\tfrac{3}{d-1}\big)
+\tfrac{\max\{4-d,0\}}{d-1}\,,\;\;\;\;\;
n\;\; {\rm odd}\,.
\end{aligned}\right.
\end{equation}
\end{theorem}
%We note that the gain with respect to the existence time given by the local theory, $O(\eps^{-n+2}$ is more significative for $d=2$ and $d=3$. Then, for $d\geq4$, $T_\eps=O(\eps^{A(-n+2)^+})$ with $A=1+\tfrac{3}{d-1}$. \\
Originally, the beam equation has been introduced  
in physics to model the oscillations of a uniform beam, 
so in a one dimensional context. 
In dimension $2$, similar equations can be used to model 
the motion of a clamped plate (see for instance the introduction of \cite{Pau}). 
In larger dimension ($d\geq3$) we do not claim 
that the beam equation \eqref{eq:beam} has a physical 
interpretation but nevertheless remains an interesting mathematical 
model of dispersive PDE.
We note that when the equation is posed on a torus, 
there is no physical reason to assume the torus to be rational.

This problem of extending solutions of semi-linear PDEs 
beyond the time given by local existence theory 
has been considered many times in the past, 
starting with Bourgain \cite{Bou96}, Bambusi \cite{Bam03} 
and Bambusi-Gr\'ebert \cite{BG} in which the authors prove 
the almost global existence for the Klein Gordon equation:
\begin{equation}\label{eq:KG}
\left\{ \begin{aligned}
& \partial_{tt}\psi-\Delta \psi+m\psi+f(\psi)=0 \,,\\
&\psi(0,x)=\psi_0\,,\\
&\partial_{t}\psi(0,x)=\psi_1\,,
 \end{aligned}\right.
 \end{equation}
on a one dimensional torus. Precisely, they proved that, given $N\geq1$, 
if the initial datum has a size $\eps$ small enough in 
$H^{s}(\mathbb{T})\times H^{s-1}(\mathbb{T})$, 
and if the mass stays outside an exceptional subset of zero measure, 
the solution of \eqref{eq:KG} exists at least on an interval of length 
$O(\eps^{-N})$. This result has been extended 
to equation \eqref{eq:KG} on Zoll manifolds (in particular spheres) 
by Bambusi-Delort-Gr\'ebert-Szeftel \cite{BDGS} but also 
for the nonlinear Schr\"odinger equation posed on $\mathbb{T}^{d}$ 
(the square torus of dimension d) \cite{BG,FG} or on 
$\R^d$ with a harmonic potential \cite{GIP}. 
What all these examples have in common 
is that the spectrum of the linear part of the equation 
can be divided into clusters that are well separated from each other. 
Actually if you considered \eqref{eq:beam} with a generic mass $m$  on the square torus $\mathbb T^d$ 
then the spectrum of $\sqrt{\Delta^2+m}$ 
(the square root comes from the fact that the equation is of order two in time) 
is given by $\{\sqrt{|j|^4+m}\mid j\in\Z^d\}$ 
which can be divided in clusters around each integers $n$  
whose diameter decreases with $|n|$. 
Thus for $n$ large enough these clusters are separated by $1/2$. 
So in this case also we could easily prove, following \cite{BG}, 
the almost global existence of the solution. 

\noindent
On the contrary when the equation is posed on an irrational torus,  
the nature of the spectrum  drastically changes: 
the differences between couples of eigenvalues 
accumulate to zero.
Even for the Klein Gordon equation \eqref{eq:KG} 
posed on $\mathbb T^d$ for $d\geq2$ the linear spectrum is not well separated. 
In both cases we could expect exchange of energy between high Fourier 
modes and thus the almost global existence 
in the sense described above is not reachable (at least up to now!). 
Nevertheless it is possible to go beyond the time given 
by the local existence theory. 
In the case of \eqref{eq:KG} on $\mathbb T^d$ for $d\geq2$, 
this local time has been extended by Delort \cite{Delort-Tori} and 
then improved in different ways by  Fang and Zhang \cite{fang}, 
Zhang \cite{Zhang} and Feola-Gr\'ebert-Iandoli \cite{FGI20} 
(in this last case a quasi linear Klein Gordon equation is considered). 
We quote also the remarkable work on multidimensional periodic water wave
by Ionescu-Pusateri \cite{IPtori}.
%For more result about the quasilinear case we mention \cite{Delort-Szeft1,Delort-circle}.

The beam equation has already been considered 
on irrational torus in dimension $2$ by R. Imekraz in \cite{Im}. 
In the case he considered, the irrationality parameter $\nu$ 
was diophantine and fixed, but a mass $m$  was added in the game 
(for us $m$ is fixed and for convenience we chose $m=1$). 
For almost all mass, Imekraz obtained a lifespan 
$T_\eps=O(\eps^{-\frac54(n-2)^+})$  
while we obtain, for almost all $\nu$, 
$T_\eps=O(\eps^{-4(n-2)^+})$ when 
$n$ is even and $T_\eps=O(\eps^{-4(n-2)-2^+})$ when $n$ is odd.

We notice that applying the Theorem 3 of \cite{BFG} (and its Corollary 1) we obtain the almost global existence for \eqref{eq:beam} on irrational tori up to a large but finite loss of derivatives.

Let us also mention some recent results about the 
longtime existence for  periodic  water waves \cite{BD,BFP,BFF1,BFP1}. 
In the same spirit we quote the long time existence 
for a general class of quasi-linear \emph{Hamiltonian} equations 
 \cite{FI20} and quasi-linear \emph{reversible} Schr\"odinger equations
 \cite{Feola-Iandoli-Long}  on the circle.
The main theorem in \cite{FI20} applies also for quasi-linear 
perturbations of the beam equation.
We mention also \cite{DI},  here the authors study 
the lifespan of small solutions of the  
semi-linear Klein-Gordon equation posed  
on a general compact boundary-less Riemannian manifold.

All previous results (\cite{Delort-Tori,fang,Zhang,FGI20,Im})
have been  obtained by a modified energy procedure. 
Such procedure partially destroys  the algebraic 
structure %\footnote{Actually this is not necessarily true. . . } 
of the equation and, thus, it makes more involved
%and prohibits 
to iterate the procedure\footnote{Actually there are papers in which such 
procedure is iterated. We quote for instance \cite{Delort-sphere} 
and reference therein.}.
On the contrary, in this paper, we begin by a Birkhoff normal form procedure 
(when $d=2,3$) before applying a modified energy step. 
Further in dimension $2$ we can iterate two steps of 
Birkhoff normal form and  therefore we get a much better time. 
The other key tool that allows us to go further in time is 
an estimate of small divisors that we have tried to optimize to the maximum: 
essentially small divisors  make us lose $(d-1)$ derivatives  
(see Proposition \ref{mainNR}) which explains the strong dependence 
of our result on the dimension $d$ of the torus and also explains why we obtain a better result than \cite{Im}.  
In section \ref{scheme} we detail the scheme 
of the proof of Theorem \ref{thm-main}.

%\noindent
%{\bf Notation.} 
%We shall use the notation $A\lesssim B$ to denote 
%$A\le C B$ where $C$ is a positive constant
%depending on  parameters fixed once for all, 
%for instance $d$, $n$. 
%We will emphasize by writing $\lesssim_{q}$ 
%when the constant $C$ 
%depends on some other parameter $q$.

\subsection{Hamiltonian formalism}
We denote by $H^{s}(\mathbb{T}^{d};\mathbb{C})$
the usual Sobolev space of functions 
$\mathbb{T}^{d}\ni x \mapsto u(x)\in \mathbb{C}$.
We expand a function $ u(x) $, $x\in \mathbb{T}^{d}$, 
 in Fourier series as 
\begin{equation}\label{complex-uU}
u(x) = \frac{1}{({2\pi})^{d/2}}
\sum_{n \in \Z^{d} } {u}_ne^{\ii n\cdot x } \, , \qquad 
{u}_n := \frac{1}{(2\pi)^{d/2}} 
\int_{\mathbb{T}^{d}} u(x) e^{-\ii n\cdot x } \, dx \, .
\end{equation}
We also use the notation
\begin{equation}\label{notaFou}
u_n^{+1} := u_n \qquad  {\rm and} 
\qquad  u_n^{-1} := \overline{u_n}  \, . 
\end{equation}
We set $\langle j \rangle:=\sqrt{1+|j|^{2}}$ for $j\in \mathbb{Z}^{d}$.
We endow $H^{s}(\mathbb{T}^{d};\mathbb{C})$ with the norm 
\begin{equation}\label{Sobnorm}
\|u(\cdot)\|_{H^{s}}^{2}:=\sum_{j\in \mathbb{Z}^{d}}\langle j\rangle^{2s}| u_{j}|^{2}\,.
\end{equation}

Moreover, for $r\in\R^{+}$, we 
denote by $B_{r}(H^{s}(\mathbb{T}^{d};\mathbb{C}))$
the ball of $H^{s}(\mathbb{T}^{d};\mathbb{C}))$ 
with radius $r$
centered at the origin.
We shall also write the norm in \eqref{Sobnorm} as
$\|u\|^{2}_{H^{s}}=
(\langle D\rangle^{s}u,\langle D\rangle^{s} u)_{L^{2}}$, 
where $\langle D\rangle e^{\ii j\cdot x}=\langle j\rangle  e^{\ii j\cdot x}$,
for any $j\in\mathbb{Z}^{d}$.

%\subsubsection{Beam equation in the complex variables}\label{s_complex} 
In the following it will be more convenient to rescale the equation \eqref{eq:beam}
and work on squared tori $\mathbb{T}^{d}$.
For any $y\in \mathbb{T}_\nu^{d}$ we write $\psi(y)=\phi(x)$
with $y=(x_1\nu_1,\ldots, x_d\nu_d)$ and $x=(x_1,\ldots,x_d)\in \mathbb{T}^{d}$.
The beam equation in \eqref{eq:beam} reads
\begin{equation}\label{eq:beamtori}
\pa_{tt}\phi+\Omega^{2}\phi+f(\phi)=0
\end{equation}
where $\Omega$ 
is the Fourier multiplier defined by linearity as
\begin{equation}\label{omegoneBeam}
\Omega e^{\ii j\cdot x}=\omega_{j} e^{\ii j\cdot x}\,,
\qquad 
\omega_{j}:=\sqrt{|j|_a^{4}+1}\,,
\quad
|j|_{a}^{2}:=\sum_{i=1}^{d}a_{i}|j_i|^{2}\,,
\;\;\;
a_{i}:=\nu_i^{2}\,,
\;\;\;\forall \,j\in \mathbb{Z}^{d}\,.
\end{equation}
Introducing the variable $v=\dot{\phi}=\pa_{t}\phi$ we can 
rewrite equation \eqref{eq:beamtori} as 
\begin{equation}\label{eq:beam2}
\dot{\phi}=-v\,,\qquad \dot{v}=\Omega^{2}\phi+f(\phi)\,.
%\left\{\begin{aligned}
%&\dot{\phi}%=-\pa_{v}\widetilde{H}(v,\phi)
%=-v\,,\\
%&\dot{v}%=\pa_{\phi}\widetilde{H}(v,\phi)
%=\Omega^{2}\phi+f(\phi)\,.
%\end{aligned}\right.
\end{equation}
By \eqref{hypFF} we note 
that \eqref{eq:beam2} can be written in the Hamiltonian form
\[
\pa_{t}\vect{\phi}{v}=X_{H_{\mathbb{R}}}(\phi,v)=J\left(
\begin{matrix}
\pa_{\phi}H_{\mathbb{R}}(\phi,v)\\
\pa_{v}H_{\mathbb{R}}(\phi,v)
\end{matrix}\right)\,,
\quad J=\sm{0}{1}{-1}{0}
\]
where $\pa$ denotes the $L^{2}$-gradient of the
Hamiltonian function
\begin{equation}\label{BeamRealHam}
 H_{\mathbb{R}}(\phi,v)=
\int_{\mathbb{T}^{d}}\big(\frac{1}{2}v^{2}+\frac{1}{2}(\Omega^{2}\phi) \phi
+F(\phi)\big)dx\,,
\end{equation}
on the phase space 
$H^{2}(\mathbb{T}^{d};\mathbb{R})\times L^{2}(\mathbb{T}^{d};\mathbb{R})$.
Indeed we have 
\begin{equation}\label{eq:1.14bis}
\mathrm{d}H_{\mathbb{R}}(\phi,v)\vect{\hat{\phi}}{\hat{v}}=-
\lambda_{\mathbb{R}}(X_{H_{\mathbb{R}}}(\phi,v),
\vect{\hat{\phi}}{\hat{v}})
\end{equation}
for any $(\phi,v), (\hat{\phi},\hat{v})$ in 
$H^{2}(\mathbb{T}^{d};\mathbb{R})\times L^{2}(\mathbb{T}^{d};\mathbb{R})$, 
where $\lambda_{\mathbb{R}}$ 
is the non-degenerate symplectic form
\[
\lambda_{\mathbb{R}}(W_1,W_2):=\int_{\mathbb{T}^{d}}(\phi_1v_2-v_1\phi_2)dx\,,
\qquad W_1:=\vect{\phi_1}{v_1}\,, W_2:=\vect{\phi_2}{v_2}\,.
\]
The Poisson bracket between two Hamiltonian 
$H_{\mathbb{R}}, G_{\mathbb{R}}: 
H^{2}(\mathbb{T}^{d};\mathbb{R})\times L^{2}(\mathbb{T}^{d};\mathbb{R})\to \mathbb{R}$
are defined as 
\begin{equation}\label{realpoisson}
\{H_{\mathbb{R}},G_{\mathbb{R}}\}
=\lambda_{\mathbb{R}}(X_{H_{\mathbb{R}}},X_{G_{\mathbb{R}}})\,.
\end{equation}

\noindent
We define the complex variables
\begin{equation}\label{beam5}
\vect{u}{\bar{u}}:=\mathcal{C}\vect{\phi}{v}\,,\quad \mathcal{C}:=
\frac{1}{\sqrt{2}}\left(
\begin{matrix}
\Omega^{\frac{1}{2}} & \ii \Omega^{-\frac{1}{2}}\\
\Omega^{\frac{1}{2}} & -\ii \Omega^{-\frac{1}{2}}
\end{matrix}
\right)\,,
\end{equation}
where $\Omega$ is the Fourier multiplier defined in \eqref{omegoneBeam}.
%\begin{equation}\label{beam5}
%u:=\frac{1}{\sqrt{2}}\big( \Omega^{\frac{1}{2}}\phi+\ii \Omega^{-\frac{1}{2}}v\big)\,.
%\end{equation}
Then the system \eqref{eq:beam2} reads
\begin{equation}\label{eq:beamComp}
\dot{u} =\ii \Omega u+ \frac{\ii }{\sqrt 2}
\Omega^{-1/2}f\left(\Omega^{-1/2}\left(\frac{u+
\bar u}{\sqrt 2}\right)\right)\,.
  \end{equation}
%On the space $L^{2}(\mathbb{T};\mathbb{C}^{2})\cap\mathcal{U}$ with 
%$\mathcal{U}:=\big\{(u^{+},u^{-})\in 
%L^{2}(\mathbb{T};\mathbb{C}^{2}) \,: \, u^{-}=\overline{u^{+}}\big\}
%$
%we consider the non-degenerate two-form
%\begin{equation}\label{dueforma}
%\lambda(U,V)=\int_{\mathbb{T}^{d}}U\cdot \ii J Vdx
%=\int_{\mathbb{T}^{d}}\ii (u \bar{v}-\bar{u}v)d x\,,\qquad 
%J=\sm{0}{1}{-1}{0}\,, 
%\quad
%U=\vect{u}{\bar{u}}\,,\quad V=\vect{v}{\bar{v}}\,.
%\end{equation}
Notice that \eqref{eq:beamComp} can be written in the Hamiltonian form
%Then equation \eqref{eq:beamComp}
%becomes a Hamiltonian system 
\begin{equation}\label{grad-ham}
%\dot{U}
\pa_{t}\vect{u}{\bar{u}}=X_{H}(u)=\ii J \left(\begin{matrix} \pa_{{u}}H(u)\\
 \pa_{\bar{u}}H(u)\end{matrix}\right)=
\left(\begin{matrix}\ii \pa_{\bar{u}}H(u)\\
-\ii \pa_{{u}}H(u)\end{matrix}\right)\,,\quad J=\sm{0}{1}{-1}{0}
%\qquad U=\vect{u}{\bar{u}}
\end{equation}
 with Hamiltonian function (see \eqref{BeamRealHam})
 \begin{equation}\label{beamHam}
 H(u)=H_{\mathbb{R}}(\mathcal{C}^{-1}\vect{u}{\bar{u}})
 =\int_{\mathbb{T}^{d}}\bar{u}\Omega u\ \mathrm{d}x 
 +\int_{\mathbb{T}^{d}} F\Big( \frac{\Omega^{-1/2}(u+\bar{u})}{\sqrt{2}}\Big)\ \mathrm{d}x
 \end{equation}
 and where $\pa_{\bar{u}}=(\pa_{\Re u}+\ii \pa_{\Im u})/2$,
 $\pa_{u}=(\pa_{\Re u}-\ii \pa_{\Im u})/2$. 
 Notice that
 \begin{equation}\label{complexVecH}
 X_{H}=\mathcal{C}\circ X_{H_{\mathbb{R}}}\circ\mathcal{C}^{-1}
 \end{equation}
 and that (using \eqref{beam5})
 \begin{equation}\label{diffHam}
\begin{aligned}
 \mathrm{d}H(u)\vect{h}{\bar{h}}=(\mathrm{d}H_{\mathbb{R}})(\phi,v)[\mathcal{C}^{-1}\vect{h}{\bar{h}} ]
 \stackrel{\eqref{eq:1.14bis},\eqref{complexVecH}}{=}
% -\lambda_{\mathbb{R}}(X_{H_{\mathbb{R}}}(\phi,v),
%\mathcal{C}^{-1}\vect{h}{\bar{h}})
%\stackrel{\eqref{complexVecH}}{=}
-\lambda(X_{H}(u),\vect{h}{\bar{h}})
\end{aligned}
 \end{equation}
 for any $h\in H^{2}(\mathbb{T}^{d};\mathbb{C})$
 and where the two form $\lambda$ is given by the push-forward 
$\lambda=\lambda_{\mathbb{R}}\circ\mathcal{C}^{-1}$.
 In complex variables the Poisson bracket in \eqref{realpoisson}
 reads
 %  and $\pa$ denotes the $L^{2}$-gradient.
 %The Poisson  bracket between two Hamiltonian functions 
%$H,G$
 %: H^{1}(\mathbb{T}; \mathbb{C}^{2})\cap\mathcal{U}\to \mathbb{R}$ 
\begin{equation}\label{Poissonbrackets}
\begin{aligned}
\{H,G\}
%:=\lambda(X_{H},X_{G})=\int_{\mathbb{T}^{d}}\ii J\nabla H\cdot \nabla Gdx
&:=\lambda(X_{H},X_{G})
=\ii \int_{\mathbb{T}^{d}}\pa_{u}G\pa_{\bar{u}}H-
\pa_{\bar{u}}G\pa_{u}H \mathrm{d}x\,,
%\\&
%=\ii \sum_{k\in\mathbb{Z}^{d} }
%\frac{\partial H}{\partial \bar{u}_{k}}\frac{\partial G}{\partial u_{k}}-
%\frac{\partial H}{\partial {u}_{k}}\frac{\partial G}{\partial \bar{u}_{k}}\,,
\end{aligned}
\end{equation}
where we set $H=H_{\mathbb{R}}\circ\mathcal{C}^{-1}$,  
$G=G_{\mathbb{R}}\circ\mathcal{C}^{-1}$.
Let us introduce an additional notation:
\begin{definition} If $j \in (\mathbb{Z}^d)^r$ for some 
$r\geq k$ then $\mu_k(j)$ denotes the 
$k^{st}$ largest number among $|j_1|, \dots, |j_r|$ 
(multiplicities being taken into account). 
If there is no ambiguity we denote it only with $\mu_k$.
\end{definition}
%\begin{example} $\mu_2((1,1),(1,1),(1,0)) = \sqrt{2}$.\end{example}
Let $r\in \mathbb{N}$, $r\geq n$.
A Taylor expansion of the Hamiltonian $H$  in \eqref{beamHam} leads to
\begin{equation}\label{expinitial}
H=Z_2+\sum_{k=n}^{r-1} H_k +R_r
\end{equation}
where 
\begin{equation}\label{quadratHam}
Z_{2}:=\int_{\mathbb{T}^{d}} \bar{u}\Omega u\ \mathrm{d}x
\stackrel{(\ref{omegoneBeam})}{=}
\sum_{j\in \mathbb{Z}^{d}}\omega_{j}|u_{j}|^{2}
\end{equation}
and 
$H_k$, $k=3,\cdots,r-1$, is an homogeneous polynomial of order $k$ of the form
\begin{equation}\label{H8}
H_k =
\sum_{\substack{\sigma\in\{-1,1\}^k,\ j\in(\Z^d)^k\\ 
\sum_{i=1}^k\sigma_i j_i=0}}(H_{k})_{\sigma,j}
u_{j_1}^{\sigma_1}\cdots u_{j_k}^{\sigma_k}\end{equation}
with 
(noticing that the zero momentum condition 
$\sum_{i=1}^k\sigma_i j_i=0$ implies $\mu_1(j)\lesssim \mu_2(j)$)
\begin{equation}\label{initCoeff}
|(H_{k})_{\sigma,j}|\lesssim_k \frac1{\mu_1(j)^{2}}\,,
\quad \forall \sigma\in\{-1,1\}^k,\ j\in(\Z^d)^k
\end{equation}
and
\begin{equation}\label{R8}
\| X_{R_r}(u)\|_{H^{s+2}}\lesssim_s \| u\|_{H^s}^{r-1}\,,
\qquad
\forall\, u\in B_{1}(
H^{s}(\mathbb{T}^{d};\mathbb{C})) \,.
\end{equation}
The estimate above follows by Moser's composition theorem in \cite{Mos66},
section 2.
Estimates \eqref{initCoeff} and \eqref{R8} 
express the regularizing effect of the 
semi-linear nonlinearity in the Hamiltonian writing of \eqref{eq:beamtori}.

\subsection{Scheme of the proof of Theorem \ref{thm-main}}\label{scheme}
 As usual Theorem  \ref{thm-main} will be proved 
 by a bootstrap argument and thus 
we want to control, 
$N_s(u(t)):=\|u(t)\|^2_{H^s}$, for $t\mapsto u(t,\cdot)$ 
a small solution 
(whose local existence is given by the standard theory for semi-linear PDEs) 
of the Hamiltonian system generated by $H$ given by \eqref{expinitial} 
for the longest time possible 
(and at least longer than the existence time given by the local theory). 
So we want to control its derivative with respect to $t$. 
We have
\begin{equation}\label{?}
\frac{d}{dt}N_s(u)=\{N_s,H\}=\sum_{k=n}^{r-1}\{N_s, H_k\} +\{N_s,R_r\}\,.
\end{equation}
By \eqref{R8} we have $\{N_s,R_r\}\lesssim \|u\|^{r-1}_{H^s}$ 
and thus we can neglect this term choosing $r$ large enough. 
Then we define  $H^{\leq N}_k$ the truncation of $H_k$ at order $N$:
\begin{equation*}\label{H8bis}
H_k^{\leq N} =
\sum_{\substack{\sigma\in\{-1,1\}^k,\ j\in(\Z^d)^k\\ 
\sum_{i=1}^k\sigma_i j_i=0,\ \mu_2(j)\leq N}}
(H_{k})_{\sigma,j}u_{j_1}^{\sigma_1}\cdots u_{j_k}^{\sigma_k}
\end{equation*}
and we set $H^{> N}_k=H_k-H^{\leq N}_k$.
As a consequence of  \eqref{initCoeff}
we have 
$\{N_s,H_k^{>N}\}\lesssim N^{-2}\|u\|^{k-1}_{H^s}$ 
and thus we can neglect these terms choosing $N$ large enough. 
So it remains to take care of 
$\sum_{k=n}^{r-1}\{N_s, H_k^{\leq N}\}$.

\noindent
The natural idea to eliminate $H_k^{\leq N}$ consists in using a 
Birkhoff normal form procedure (see \cite{BG,G}). 
In order to do that,  we have first to solve the homological equation
\[
\{\chi_k,Z_2\}+H_k^{\leq N}=Z_k\,.
\] 
This is achieved in Lemma \ref{lem:homoeq} and, 
thanks to the control of the small divisors given by Proposition \ref{mainNR}, 
we get that there exists $\alpha\equiv \alpha(d,k)>0$ 
such that for any $\delta>0$
\begin{equation}\label{chiCoeff}
|(\chi_{k})_{\sigma,j}|
\lesssim_\delta {\mu_1(j)^{d-3+\delta}}\mu_3(j)^\alpha\,,
\quad 
\forall \sigma\in\{-1,1\}^k,\ j\in(\Z^d)^k\,.
\end{equation}
From \cite{BG} we learn that the positive power of 
$\mu_3(j)$ appearing in the right hand side of \eqref{chiCoeff} 
is not dangerous\footnote{When you have a control of the small 
divisors involving only $\mu_3(j)$ then you can solve the homological 
equation at any order and you obtain an almost 
global existence result in the spirit of \cite{BG}. 
This would be the case if we consider the semi-linear 
beam equation on the squared torus $\T^d$.} 
(taking $s$ large enough) but the positive power of 
$\mu_1(j)$ implies a loss of derivatives. 
So this step can be achieved only assuming $d\leq 3$ 
and in that case the corresponding flow is well defined in 
$H^s$ (with $s$ large enough) and is controlled by 
$N^{\delta}$ (see Lemma \ref{lem:flowmap}). 
In other words, this step is performed only when $d=2,3$, 
when $d\geq 4$ we directly go to the modified energy step. 

\noindent
For $d=2,3$, let us focus on $n=3$. 
After this Birkhoff normal form step, we are left  with  
\[
H\circ \Phi_{\chi_3}=Z_2+Z_3+Q_4+\rm{negligible\ terms}
\]
where $Q_4$ is a Hamiltonian of order $4$ 
whose coefficients are bounded by $\mu_1(j)^{d-3+\delta}$ 
(see Lemma \ref{lem:hamvec}, estimate \eqref{est:smooth-vec-fie}) 
and $Z_3$ is a Hamiltonian of order $3$ 
which is resonant: $\{Z_2,Z_3\}=0$. 
Actually, as consequence Proposition \ref{mainNR}, 
$Z_3=0$ and thus we have eliminated all the terms of order $3$ in \eqref{?}. 

\noindent
In the case $d=2$, $Q_4^{\leq N}$ is still 
$(1-\delta)$-regularizing and we can perform a 
second Birkhoff normal form. 
Actually, since in eliminating $Q_4^{\leq N}$ 
we create terms of order at least $6$, 
we can eliminate both $Q_4^{\leq N}$ and $Q_5^{\leq N}$. 
So, for $d=2$, we are left with
\begin{equation*}\label{H6} 
\tilde{H}=H\circ \Phi_{\chi_3}\circ \Phi_{\chi_4+\chi_5}
=Z_2+ Z_4+Q_6+\rm{negligible\ terms}
\end{equation*}
where $Z_4$ is Hamiltonian of order $4$ 
which is resonant\footnote{Notice that there is no resonant term 
of odd order by Proposition \ref{mainNR}, 
in other words $Z_3=Z_5=0$.}, 
$\{Z_2,Z_4\}=0$, and $Q_6$ is a Hamiltonian of order $6$ 
whose coefficients are bounded by $N^{2\delta}$. 
Since resonant Hamiltonians commute with $N_s$, 
the first contribution in \eqref{?} is  $\{N_s,Q_6\}$.  
This is essentially the statement of Theorem \ref{BNFstep} 
in the case $d=2$ and $n=3$ and 
this achieves the Birkhoff normal forms step.

\medskip

Let us describe the modified energy step only in the case 
$d=2$ and $n=3$ and let us focus on the worst term in 
$\{N_s, \tilde H\}$, i.e. $\{N_s,Q_6\}$. 
Let us write
\[
Q_6=
\sum_{\substack{\sigma\in\{-1,1\}^6,\ j\in(\Z^d)^k\\ 
|j_1|\geq \cdots\geq |j_6|\\
\sum_{i=1}^6\sigma_i j_i=0}}
(Q_6)_{\sigma,j}u_{j_1}^{\sigma_1}\cdots u_{j_6}^{\sigma_6}\,.
\]
From Proposition \ref{mainNR} we learn that if  $\sigma_1\sigma_2=1$ 
then the small divisor associated with $(j,\sigma)$ 
is controlled  by $\mu_3(j)$ 
and thus we can eliminate the corresponding monomial 
by one more Birkhoff normal forms step\footnote{In fact in 
section \ref{section:modif}, for the sake of simplicity, 
we prefer to apply a modified energy strategy 
to all the terms of $Q_6$ (see also Remark \ref{choice}).}. 
Now if we assume $\sigma_1\sigma_2=-1$ we have
\begin{align*}
|\{N_s,u_{j_1}^{\sigma_1}\cdots u_{j_6}^{\sigma_6}\}|
&=
|\sum_{i=1}^6 \sigma_{j_i}\langle j_i\rangle^{2s}|
|u_{j_1}^{\sigma_1}\cdots u_{j_6}^{\sigma_6}|
\\&
\leq   (\langle j_1\rangle^{2s}-\langle j_2\rangle^{2s}+ 4 \langle j_3\rangle^{2s})
|u_{j_1}^{\sigma_1}\cdots u_{j_6}^{\sigma_6}|
\\&
\leq \big(s(\langle j_1\rangle^{2}-\langle j_2\rangle^{2})\langle j_1\rangle^{2(s-1)}
+ 4 \langle j_3\rangle^{2s}\big)
|u_{j_1}^{\sigma_1}\cdots u_{j_6}^{\sigma_6}|
\\&
\lesssim_s (\langle j_1\rangle^{2s-1}\langle j_3\rangle
+ 4 \langle j_3\rangle^{2s})|u_{j_1}^{\sigma_1}\cdots u_{j_6}^{\sigma_6}|
\\&
\lesssim_s \mu_1^{-1}\|u\|_{H^s}^6
\end{align*}
where we used the zero momentum condition, 
$\sum_{i=1}^6\sigma_i j_i=0$, to obtain $|j_1-j_2|\leq 4|j_3|$. This gain of one derivative, also known as the commutator trick, is central in a lot of  results about modified energy \cite{Delort-Tori,BFG} or growth of Sobolev norms \cite{Bou99, Del10, Ber,BGMR}.\\
So if $Q_6^-$ denotes the restriction of $Q_6$ to monomials satisfying 
$\sigma_1\sigma_2=-1$ we have essentially proved that
\[
|\{N_s,Q_6^{-,>N_1}\}|\lesssim  N_1^{-1}\|u\|_{H^s}^6\,.
\]
Then we can consider the modified  energy 
$N_s+E_6$ with $E_6$ solving 
\[
\{E_6,Z_2\}=\{N_s, Q_6^{-,\leq N_1}\}
\]
in such a way that
\[
\{N_s+E_6,\tilde H\}=\{N_s, Q_6^{-,> N_1}\}+\{N_s,\tilde H_7\}
+\{E_6,Z_4\}+\rm{negligible\ terms}\,.
\]
Since this modified energy will not produce new terms of order $7$, 
we can in the same time eliminate $Q_7^{-,\leq N_1}$.
Thus we obtain a new energy,  $N_s+E_6+E_7$, 
which is equivalent to $N_s$ in a neighborhood of the origin, 
and such that, by neglecting all the powers of $N^\delta$ and $N_1^\delta$ 
which appear when we work carefully 
(see \eqref{energyestimate} for a precise estimate),
\[
|\{N_s+E_6+E_7,\tilde H\}|\lesssim_s 
N_1^{-1}\|u\|_{H^s}^6+ \|u\|_{H^s}^8+N^{-1}\|u\|_{H^s}^3\,.
\]
Then, a suitable choice of $N$ and $N_1$ and a standard bootstrap argument  lead to, 
$T_\eps=O(\eps^{-6})$ by using this rough estimate, and 
$T_\eps=O(\eps^{-6^-})$ by using the precise estimate 
(see section \ref{section:boot}).

\medskip

\begin{remark}\label{choice}
In principle a Birkhoff normal form procedure gives more 
than just the control of $H^s$ norm of the solutions, 
it gives an equivalent Hamiltonian system and therefore potentially 
more information about the dynamics of the solutions. 
However, if one wants to control only the solution in $H^s$ norm, 
the modified energy method is sufficient and simpler. 
One could therefore imagine applying this last method from the beginning. 
However, when we iterate it, the modified energy method brings up terms that, 
when we apply a Birkhoff procedure, turn out to be zero. 
Unfortunately we have not been able to prove the cancellation of these terms 
directly by the modified energy method, that is why we use successively a Birkhoff normal form procedure and a modified energy procedure.
 %hence the two-step decomposition of our proof.
\end{remark}

\noindent
{\bf Notation.} 
We shall use the notation $A\lesssim B$ to denote 
$A\le C B$ where $C$ is a positive constant
depending on  parameters fixed once for all, 
for instance $d$, $n$. 
We will emphasize by writing $\lesssim_{q}$ 
when the constant $C$ 
depends on some other parameter $q$.

\section{Small divisors}\label{sd}
As already remarked in the introduction, the proof of Theorem \ref{thm-main}
is based on a normal form approach. In particular we have to deal with 
a small divisors problem involving linear combination of linear frequencies 
$\omega_{j}$ in \eqref{omegoneBeam}.

This section is devoted to establish suitable lower bounds 
for \emph{generic} (in a probabilistic way) 
choices of the parameters $\nu$ excepted for exceptional indices 
for which the small divisor is identically zero.
According to the following definition such indices are called \emph{resonant}.

\begin{definition}[Resonant indices]
\label{def_res}
Being given $r\geq 3$, $j_1,\dots,j_r \in \mathbb{Z}^d$ and $\sigma_1,\dots,\sigma_r\in \{-1,1\}$, the couple $(\sigma,j)$ is resonant if 
$r$ is even and there exists a permutation $\rho \in \mathfrak{S}_r$ such that 
$$
\forall k\in \llbracket 1,r/2 \rrbracket,  \ \begin{pmatrix} |j_{\rho_{2k-1},1}| \\ \vdots \\ |j_{\rho_{2k-1},d}| \end{pmatrix} =\begin{pmatrix} |j_{\rho_{2k},1}| \\ \vdots \\ |j_{\rho_{2k},d}| \end{pmatrix} \quad \mathrm{and} \quad \sigma_{\rho_{2k-1}} = -\sigma_{\rho_{2k}}.
$$
\end{definition}

In this section we aim at proving the following proposition whose proof is postponed to the end of this section (see subsection \ref{sub_proof_mainNR}). We recall that $a$ is defined with respect to the length, $\nu$, of the  torus by the relation $a_i= \nu_i^2$ (see \eqref{omegoneBeam}).
\begin{proposition} \label{mainNR} For almost all $a\in (1,4)^d$, there exists $\gamma>0$ such that for all $\delta>0$, $r\geq 3$, $\sigma_1,\dots,\sigma_r\in \{-1,1\}$, $j_1,\dots,j_r \in \mathbb{Z}^d$ satisfying $\sigma_1 j_1+\dots+\sigma_r j_r = 0$ and $|j_1|\geq \dots \geq |j_r|$ at least one of the following assertion holds
\begin{enumerate}[(i)]
\item $(\sigma,j)$ is resonant (see Definition \ref{def_res})
\item $\sigma_1 \sigma_2 = 1$ and 
$$
\left| \sum_{k=1}^r \sigma_k \sqrt{1+|j_k|_a^4} \right| \gtrsim_r \gamma\, (\langle j_3 \rangle \dots \langle j_r \rangle)^{-9dr^2},
$$
\item $\sigma_1 \sigma_2 = -1$ and
$$
\left| \sum_{k=1}^r \sigma_k \sqrt{1+|j_k|_a^4} \right| \gtrsim_{r,\delta} \gamma\, \langle j_1 \rangle^{-(d-1+\delta)} (\langle j_3 \rangle \dots \langle j_r \rangle)^{-44dr^4}.
$$
\end{enumerate}
\end{proposition}
We refer the reader to Lemma \ref{eq_assump} and its corollary to understand how we get this degeneracy with respect to $j_1$.

\subsection{A weak non-resonance estimate}

In this subsection we aim at proving the following technical lemma.
\begin{lemma}
\label{lem_gen}  If $r\geq 1$, $(j_1,\dots,j_r) \in (\mathbb{N}^d)^r$ is injective\footnote{i.e. $\forall k,\ell\in \llbracket 1,r \rrbracket, \ k\neq \ell \Rightarrow \ j_{k}\neq j_{\ell}$.}, $n \in (\mathbb{Z}^*)^r$ and $\kappa \in \mathbb{R}^d$ satisfies $\kappa_{i_{\star}}=0$ for some $i_{\star}\in \llbracket 1,d\rrbracket$ then we have
$$
\forall \gamma>0, \  \left|\{ a \in (1,4)^d \ : \ \big\vert \kappa \cdot a + \sum_{k=1}^r n_k \sqrt{1+|j_k|^4_a}\big\vert<\gamma \}\right| \lesssim_{r,d} \gamma^{\frac1{r(r+1)}}  (\langle j_1 \rangle \dots \langle j_r \rangle)^{\frac{12}{r+1}}.
$$
\end{lemma}
Their proofs (postponed to the end of this subsection) rely essentially on the following lemma.
\begin{lemma} \label{lem_good_form} If $I,J$ are two bounded intervals of $\mathbb{R}_+^*$, $r\geq 1$, $(j_1,\dots,j_r) \in (\mathbb{N}^d)^r$ is injective, $n \in (\mathbb{Z}^*)^r$ and $h: J^{d-1}\to \mathbb{R}$ is measurable then for all $\gamma >0$ we have
$$
 \left|\{ (m,b) \in I\times J^{d-1} \ : \ \big\vert h(b)+ \sum_{k=1}^r n_k \sqrt{m+ |j_k|^4_{(1,b)}}\big\vert<\gamma \}\right| \lesssim_{r,d,I,J} \gamma^{\frac1{r(r+1)}}  (\langle j_1 \rangle \dots \langle j_r \rangle)^{\frac{12}{r+1}}
$$
where $(1,b):=(1,b_1,\dots,b_{d-1})\in \mathbb{R}^d$.
\end{lemma}
\begin{proof}[Proof of Lemma \ref{lem_good_form}] The proof of this lemma is classical and follows the lines of \cite{Bam03}. 

Without loss of generality, we assume that $\gamma\in (0,1)$. Let $\eta\in (0,1)$ be a positive number which will be optimized later with respect to $\gamma$. If $1\leq i<k \leq r$ then we have
$$
| j_i|_{1,b}^2 - | j_k|_{1,b}^2 = (j_{i,1}^2-j_{k,1}^2)+b_1 (j_{i,2}^2-j_{k,2}^2) + \dots + b_{d-1}(j_{i,d}^2-j_{k,d}^2).
$$
Since, by assumption, $(j_1,\dots,j_r)$ is injective, either there exists $\ell \in \llbracket 2,d \rrbracket$ such that $j_{i,\ell}\neq j_{k,\ell}$ or $j_{i,1}\neq j_{k,1}$ and $j_{i,\ell}=j_{k,\ell}$ for $\ell = 2,\dots, d$. Note that in this second case, we have $|| j_i|_{1,b}^2 - | j_k|_{1,b}^2 |\geq 1$.
In any case, since the dependency with respect to $b$ is affine the set
\begin{center}
$\mathcal{P}_{\eta}^{(i,k)}  = \{ b \in J^{d-1} \ | \ | j_i|_{1,b}^2 - | j_k|_{1,b}^2 | < \eta \}$ satisfies $|\mathcal{P}_{\eta}^{(i,k)}|<\, \eta (1+|J|^{d-1})$.
\end{center}
Therefore, we have  
\begin{multline}
\label{Tyrannosaurus}
\{ (m,b) \in I\times J^{d-1}  :  \big\vert  h(b)+ \sum_{k=1}^r n_k \sqrt{m+ |j_k|^4_{(1,b)}}\big\vert<\gamma \} 
\leq \frac{r(r-1)}2  |I|\, \eta\, (1+|J|^{d-1}) \\+ |J|^{d-1} \sup_{\forall i<k, \ b \notin \mathcal{P}_{\eta}^{(i,k)} } |\{ m \in I \ : \ \big\vert  h(b)+ \sum_{k=1}^r n_k \sqrt{m+ |j_k|^4_{(1,b)}} \big\vert<\gamma \}|.
\end{multline}
In order to estimate this last measure we fix $b\in J^{d-1}\setminus \bigcup_{i<k} \mathcal{P}_{\eta}^{(i,k)}$ and we define $g:I\to \mathbb{R}$ by
$$
g(m) =  h(b)+ \sum_{k=1}^r n_k \sqrt{m+ |j_k|^4_{(1,b)}}.
$$
By a straightforward calculation, for $\ell \geq 1$, we have
\begin{equation}
\label{ecureuil} 
\partial_m^\ell g(m)  = c_{\ell} \sum_{k=1}^r n_k (m+ |j_k|^4_{(1,b)})^{\frac12-\ell} \quad \mathrm{where} \quad c_{\ell} = \prod_{i=0}^{\ell-1} \frac12 - i.
\end{equation}
Therefore, we have
$$
\begin{pmatrix}
c_{1}^{-1}  \partial_m^{1} g \\ \vdots \\ c_{r}^{-1}\partial_m^{r} g
\end{pmatrix} = \begin{pmatrix} (m+ |j_1|^4_{(1,b)})^{0} & \dots & (m+ |j_r|^4_{(1,b)})^{0} \\ \vdots & & \vdots \\  (m+ |j_1|^4_{(1,b)})^{-(r-1)} & \dots & (m+ |j_r|^4_{(1,b)})^{-(r-1)}  \end{pmatrix}  \begin{pmatrix}
n_1 \sqrt{m+ |j_1|^4_{(1,b)}}^{-1} \\ \vdots \\ n_r \sqrt{m+ |j_r|^4_{(1,b)}}^{-1}  \end{pmatrix}.
$$
Denoting by $V$ this Vandermonde matrix, by $|x|_{\infty} := \max |x_i|$ for $x\in \mathbb{R}^d$ and also by $|\cdot|_\infty$ the associated matrix norm, we deduce that
\begin{equation}
\label{grenouille}
\max_{i=1}^{r} c_{i}^{-1} |\partial_m^{i} g(m)| \geq |V^{-1}|_{\infty}^{-1} \max_{i=1}^{r} |n_i| \sqrt{m+ |j_i|^4_{(1,b)}}^{-1}.
\end{equation}
We recall that the invert of $V$ is given by
\begin{equation}
\label{lapin}
(V^{-1})_{i,\ell} = (-1)^{r-\ell} \frac{S_{r-\ell}((\frac1{m+|j_k|^4_{(1,b)}})_{k\neq i} )}{\displaystyle \prod_{k\neq i} \frac1{m+|j_i|^4_{(1,b)}}- \frac1{m+|j_k|^4_{(1,b)}} }
\end{equation}
(this formula can be easily derived using the Lagrange interpolation polynomials) where $S_{\ell} : \mathbb{R}^{r-1}\to \mathbb{R}$ is the $\ell^{st}$ elementary symmetric function
$$
S_\ell(x) = \sum_{1\leq k_1<\dots<k_{\ell} \leq r-1} x_{k_1}\dots x_{k_{\ell}} \quad \mathrm{and} \quad S_0(x):=1.
$$
Furthermore, we have
\begin{equation}
\label{asticot}
|V^{-1}|_{\infty} = \max_{i=1}^r \sum_{\ell=1}^r |(V^{-1})_{i,\ell}|.
\end{equation}
To estimate $|V^{-1}|_{\infty} $ in \eqref{grenouille}, we use the estimates
$$
S_{r-\ell}((\frac1{m+|j_k|^4_{(1,b)}})_{k\neq i} ) \lesssim_{r,J,I} 1 \quad \mathrm{and} \quad \left|\frac1{m+|j_i|^4_{(1,b)}}- \frac1{m+|j_k|^4_{(1,b)}}\right| \gtrsim_{J,I} \frac{\eta}{\langle j_k \rangle^6}.
$$
Indeed, if $||j_i|^4_{(1,b)} - |j_k|^4_{(1,b)}|\geq \frac12|j_i|^4_{(1,b)} $ we have
$$
\left|\frac1{m+|j_i|^4_{(1,b)}}- \frac1{m+|j_k|^4_{(1,b)}}\right|  = \left|\frac{|j_i|^4_{(1,b)} - |j_k|^4_{(1,b)}}{(m+|j_i|^4_{(1,b)})( m+|j_k|^4_{(1,b)})}\right|  \gtrsim_{I,J} \frac1{\langle j_k \rangle^4}
$$
and conversely, if $||j_i|^4_{(1,b)} - |j_k|^4_{(1,b)}|\leq \frac12|j_i|^4_{(1,b)}$ then $|j_i|^4_{(1,b)} \leq 2 |j_k|^4_{(1,b)}$ and so, since $b\in J^{d-1}\setminus \bigcup_{i<k} \mathcal{P}_{\eta}^{(i,k)}$, we have
$$
\left|\frac1{m+|j_i|^4_{(1,b)}}- \frac1{m+|j_k|^4_{(1,b)}}\right|  \gtrsim_{I,J} \frac{(|j_i|^2_{(1,b)}+|j_k|^2_{(1,b)})||j_i|^2_{(1,b)}-|j_k|^2_{(1,b)}| }{\langle j_k \rangle^8}  \gtrsim_{I,J}  \frac{\eta}{\langle j_k \rangle^6}.
$$
Therefore by \eqref{asticot} and \eqref{lapin}, we have
$$
|V^{-1}|_{\infty} \lesssim_{r,I,J} \eta^{-(r-1)} (\langle j_1 \rangle \dots \langle j_r \rangle)^6
$$
Consequently, we deduce from \eqref{grenouille} that 
\begin{equation}
\label{kangourou}
\max_{i=1}^{r} |\partial_m^{i} g(m)| \gtrsim_{r,I,J} \eta^{r-1} (\langle j_1 \rangle \dots \langle j_r \rangle)^{-6}  |n|_{\infty}.
\end{equation}
Furthermore, considering \eqref{ecureuil}, it is clear that
$$
|\partial_m^\ell g(m)| \lesssim_{\ell,I,J} |n|_{\infty}.
$$
As a consequence, being given $\rho>0$ (that will be optimized later), applying Lemma B.1. of \cite{Eli02}, we get $N$ sub-intervals of $I$, denoted $\Delta_{1},\dots,\Delta_{N}$ such that
$$
N \lesssim_{I,r} (\langle j_1 \rangle \dots \langle j_r \rangle)^{6} \eta^{-(r-1)},
\quad \max_{i=1}^N |\Delta_{i}| \lesssim_{I,r} \left( \frac{\rho (\langle j_1 \rangle \dots \langle j_r \rangle)^{6}}{\eta^{r-1} |n|_{\infty}} \right)^{\frac1{r-1}},
$$
$$
|\partial_m g(m)| \geq \rho \quad \forall m \in I \setminus (\Delta_1 \cup \dots \cup \Delta_N).
$$
Observing that $I \setminus (\Delta_1 \cup \dots \cup \Delta_N)$ can be written as the union of $M$ intervals with $M\lesssim 1 + N$, we deduce that
\begin{multline*}
|\{ m \in I \ : \ \big\vert  h(b)+ \sum_{k=1}^r n_k \sqrt{m+ |j_k|^4_{(1,b)}} \big\vert<\gamma \}| < M \rho^{-1} \gamma + N \max_{i=1}^N |\Delta_{i}|\\
\lesssim_{I,r}  (\langle j_1 \rangle \dots \langle j_r \rangle)^{6} \eta^{-(r-1)} \left[ \rho^{-1} \gamma +  \left( \frac{\rho (\langle j_1 \rangle \dots \langle j_r \rangle)^{6}}{\eta^{r-1} |n|_{\infty}} \right)^{\frac1{r-1}} \right].
\end{multline*}
We optimize $\rho$ to equalize the two terms in this last sum :
$$
\rho^{\frac{r}{r-1}}= \gamma \left(\frac{\eta^{r-1} |n|_{\infty}} { (\langle j_1 \rangle \dots \langle j_r \rangle)^{6}}\right)^{\frac1{r-1}}.
$$
This provides the estimate
\begin{equation*}
\begin{aligned}
%\begin{multline*}
|\{ m \in I \ : \ \big\vert  h(b)+ \sum_{k=1}^r n_k &\sqrt{m+ |j_k|^4_{(1,b)}} \big\vert<\gamma \}| 
\\ &\lesssim_{I,r} \gamma^{\frac1r}  
(\langle j_1 \rangle \dots \langle j_r \rangle)^{6} \eta^{-(r-1)}  
\left(\frac { (\langle j_1 \rangle \dots \langle j_r \rangle)^{6}}{\eta^{r-1} 
|n|_{\infty}}\right)^{\frac1{r}} 
\\ &\lesssim_{I,r} \left(\frac{\gamma}{|n|_{\infty}}\right)^{\frac1r} 
\eta^{-(r-1+\frac{r-1}r)} (\langle j_1 \rangle \dots \langle j_r \rangle)^{12}\,.
%\end{multline*}
\end{aligned}
\end{equation*}
Finally, we optimize \eqref{Tyrannosaurus} by choosing
$$
\eta =\gamma^{\frac1r} 
\eta^{-(r-1+\frac{r-1}r)} (\langle j_1 \rangle \dots \langle j_r \rangle)^{12}
$$
and, recalling that $|n|_{\infty}\geq 1$, we get
$$
\begin{aligned}
 \left|\{ (m,b) \in I\times J^{d-1} \ : \ \big\vert h(b)
 + \sum_{k=1}^r n_k \sqrt{m+ |j_k|^4_{(1,b)}}\big\vert<\gamma \}\right| &
 \\&\!\!\!\!\!\!\!\!\!\!\!\!\!\! 
 \lesssim_{r,d,I,J} \left(\gamma^{\frac1{r}}  
 (\langle j_1 \rangle \dots \langle j_r \rangle)^{12}\right)^{\frac1{r+\frac{r-1}{r}}}\,.
 \end{aligned}
$$
Since this measure is obviously bounded by $|I||J|^{d-1}$, the exponent $r+\frac{r-1}{r}$ can be replaced by $r+1$ in the above expression which conclude this proof.
\end{proof}

%%%%PROOF of Lemma General
Now using Lemma \ref{lem_good_form}, we prove Lemma \ref{lem_gen}.
\begin{proof}[Proof of Lemma \ref{lem_gen}] Without loss of generality we assume that $i_{\star} = 1$. First, since $\kappa_1=0$, we note that we have
$$
G(a):=\kappa\cdot a + \sum_{k=1}^r n_k \sqrt{1+|j_k|^4_a}= \frac1{\sqrt{m}}( h(b) + \sum_{k=1}^r n_k \sqrt{m+|j_k|^4_{(1,b)}} )=:\frac1{\sqrt{m}}F(m,b)
$$
where 
$$
m=\frac1{a_1^2}, \quad b=(\frac{a_2}{a_1},\dots,\frac{a_d}{a_1}) \quad \mathrm{and} \quad h(b) = \sum_{k=2}^r \kappa_k b_k.
$$
Let denote by $\Psi$ the map $a\mapsto (m,b)$. It is clearly smooth and injective. Furthermore, we have
$$
\det \mathrm{d}\Psi(a) = \begin{vmatrix} - 2a_1^{-3} & -a_2 a_1^{-2} & \dots & -a_d a_1^{-2} \\  & a_1^{-1} \\  & & \ddots \\  & & & a_1^{-1}   \end{vmatrix} = 2\, (-1)^d a_1^{-d-2}.
$$
Consequently, $\Psi$ is  a smooth diffeomorphism onto its image $\Psi((1,4)^d)$ which is included in the rectangle $(\frac1{16},1)\times (\frac14,4)^{d-1}$. Therefore, by a change of variable, we have
\begin{equation*}
\begin{split}
|\{ a \in (1,4)^d \ : \ \big\vert G(a) \big\vert<\gamma \}|  &= \int_{a\in (1,4)^d} \mathbb{1}_{ | G(a) |<\gamma} \mathrm{d}a \\
&= \int_{(m,b)\in \Psi((1,4)^d)} \mathbb{1}_{| F(m,b) |<\sqrt{m}\gamma} (2\, \sqrt{m}^{-d-2}) \mathrm{d}(m,b) \\
&\leq 2^{2d+5}  |\{ (m,b) \in (\frac1{16},1)\times (\frac14,4)^{d-1} \ : \ |F(m,b)|<\gamma \}|.
\end{split}
\end{equation*}
Finally, by applying Lemma \ref{lem_good_form}, we get the expected estimate.
\end{proof}

\subsection{Non-resonance estimates for two large modes}
In this subsection we consider $r\geq 3$, $(j_k)_{k\geq 3} \in (\mathbb{Z}^d)^{r-2}$ and $\sigma\in \{-1,1\}^r$ such that $\sigma_1=-\sigma_2$ as fixed. We define $j_{\geq 3} \in \mathbb{Z}^d$ by
\begin{equation}
\label{def_jgeq3}
j_{\geq 3} := \sigma_3 j_3+\dots+\sigma_r j_r.
\end{equation}
Being given $j_1\in \mathbb{Z}^d$, we define implicitly $j_2:= j_1+\sigma_1 j_{\geq 3}$ in order to satisfy the \emph{zero momentum condition}
\begin{equation}
\label{eq_zeromom}
\sum_{k=1}^r \sigma_k j_k=0,
\end{equation}
and we define the function $g_{j_1}:(1,4)^d \to \mathbb{R}$ by
$$
g_{j_1}(a) = \sum_{k=1}^n \sigma_k \sqrt{1+|j_k|_a^4}.
$$
Finally, for $\gamma >0$, we introduce the following sets
$$
\mathcal{I} = \{ i \ : \ j_{\geq 3,i}\neq 0\},  \quad
C_i = \{ j_1\in \mathbb{Z}^d \ : \ |j_{1,i}| \geq 2(1+ \sum_{k\geq 3} |j_{k,i}|^2) \}, 
$$
$$
S =  \{ j_1 \in \mathbb{Z}^d \setminus \bigcup_{i\in \mathcal{I}} C_i \ : \ (\sigma,j) \mathrm{\ is\ non}\tiret\mathrm{resonant}\footnote{see Definition \ref{def_res}}\},
$$
$$
\mathrm{and} \quad  R_{\gamma} = \{ j_1 \in S \ : \ |j_1|\geq  \gamma^{-1/2}(\langle j_3 \rangle \dots \langle j_r \rangle)^{2dr^2}\}.
$$

First, we prove the following technical lemma whose Corollary \ref{cor_gen_nice} allows to deal with the non-degenerated cases.
\begin{lemma} 
\label{lem_2l_1}
If there exists $i\in \llbracket 1,d\rrbracket$ such that 
\begin{equation}
\label{eq_assump}
|(j_{1,i}+j_{2,i}) j_{\geq 3,i}|\geq  2 (1+\sum_{k= 3}^r  j_{k,i}^2)
\end{equation}
then for all $\gamma>0$
\begin{equation}
\label{eq_conclu}
\left|\{a\in (1,4)^d \ : \ |\sum_{k=1}^r \sigma_k \sqrt{1+|j_k|^4_a} <\gamma|\}\right| < \frac{2\,\gamma}{|j_{1,i}+j_{2,i}|}.
\end{equation}
\end{lemma}
\begin{proof} Without loss of generality we assume that $\sigma_1=1$ and $\sigma_2 = -1$. We compute the derivative with respect to $a_1$
$$
\partial_{a_1} \sum_{k=1}^r \sigma_k \sqrt{1+|j_k|^4_a} =  \sum_{k=1}^r \sigma_k j_{k,i}^2 \frac{|j_k|^2_a }{\sqrt{1+|j_k|^4_a}}.
$$
Consequently, we have
$$
\left| \partial_{a_1} \left(\sum_{k=1}^r \sigma_k \sqrt{1+|j_k|^4_a}\right)  \right| \geq \left|  j_{1,i}^2 \frac{|j_1|^2_a }{\sqrt{1+|j_1|^4_a}} - j_{2,i}^2 \frac{|j_2|^2_a }{\sqrt{1+|j_2|^4_a}}  \right| - \sum_{k= 3}^r  j_{k,i}^2.
$$
Furthermore, we have
$$
\left| \frac{|j_1|^2_a }{\sqrt{1+|j_1|^4_a}} - 1  \right| \leq \frac1{2|j_1|^2_a}.
$$
Consequently, we get
$$
\left| \partial_{a_1} \left(\sum_{k=1}^r \sigma_k \sqrt{1+|j_k|^4_a}\right) \right|  \geq    |j_{1,i}^2 - j_{2,i}^2| - 1-\sum_{k= 3}^r  j_{k,i}^2.
$$
Observing that by definition we have $j_{1,i}^2 - j_{2,i}^2 = j_{\geq 3,i}(j_{1,i} + j_{2,i})$, we deduce of the assumption \eqref{eq_assump} that
$$
\left| \partial_{a_1} \sum_{k=1}^r \sigma_k \sqrt{1+|j_k|^4_a}\right| \geq \frac12 |j_{\geq 3,i}(j_{1,i} + j_{2,i})|
$$
Since by \eqref{eq_assump} we know that $j_{\geq 3,i} \in \mathbb{Z}\setminus \{0\}$, we deduce that 
$$
\left| \partial_{a_1} \sum_{k=1}^r \sigma_k \sqrt{1+|j_k|^4_a}\right| \geq \frac12 |(j_{1,i} + j_{2,i})|.
$$
Therefore $a_1\mapsto \sum_{k=1}^r \sigma_k \sqrt{1+|j_k|^4_a}$ is a diffeomorphism (it is a smooth monotonic function). Consequently, applying this change of coordinate, we get directly \eqref{eq_conclu} which conclude this proof.
\end{proof}
\begin{corollary} \label{cor_gen_nice} For all $\gamma>0$ we have
\begin{equation}
\label{est1}
\forall i \in \mathcal{I}, \ |\{ a\in (1,4)^d \ : \ \exists j_1\in C_i, \ |g_{j_1}(a)| < \gamma |j_1|^{-(d-1)} \log^{-2d}(| j_1|) \}| \lesssim_d \gamma
\end{equation}
\end{corollary}
\begin{proof}[\underline{Proof of Corollary \ref{cor_gen_nice}}] Let $j_1\in C_i$. By definition of $j_2$, we have
$$
|j_{1,i}+j_{2,i}|\geq 2|j_{1,i}| - \sum_{k=3}^r |j_{k,i}|.
$$
Consequently, since $j_1\in C_i$, we have
$$
|j_{1,i}+j_{2,i}|\geq 2 |j_{1,i}| - \sum_{k=3}^r |j_{k,i}|^2 \geq \frac32 |j_{1,i}|.
$$
Therefore, since $j_{\geq 3,i}\neq 0$, we have
$$
|j_{\geq 3,i}(j_{1,i}+j_{2,i})| \geq \frac32 |j_{1,i}| \geq 3(1+ \sum_{k\geq 3} |j_{k,i}|^2).
$$
Applying Lemma \ref{lem_2l_1}, we deduce that for all $\gamma>0$
\begin{equation*}
\left|\{a\in (1,4)^d \ : \ |\sum_{k=1}^r \sigma_k \sqrt{1+|j_k|^4_a} <\gamma|\}\right| < \frac{4\gamma}{3|j_{1,i}|}.
\end{equation*}
Consequently, we have
\begin{equation*}
\begin{split}
&|\{ a\in (1,4)^d \ : \ \exists j_1\in C_i, \ |g_{j_1}(a)| < \gamma |j_1|^{-(d-1)} \log^{-2d}(| j_1|) \}| \\
=& \left| \bigcup_{j_1\in C_i} \{ a\in (1,4)^d \ : \  |g_{j_1}(a)| < \gamma |j_1|^{-(d-1)} \log^{-2d}(| j_1|) \}  \right| \\
\leq& \sum_{j_1\in C_1} \left| \{ a\in (1,4)^d \ : \  |g_{j_1}(a)| < \gamma |j_1|^{-(d-1)} \log^{-2d}(| j_1|) \}  \right| \\
\lesssim& \, \gamma \sum_{j_1\in C_i} \frac1{|j_1|^{(d-1)} |j_{1,i}| \log^{2d}(|j_1|)} \lesssim_d \gamma.
\end{split}
\end{equation*}
\end{proof}
In the following lemma, we deal with most of the degenerated cases.
\begin{lemma} \label{lem_deg_large} For all $\gamma>0$, we have
\begin{equation}
\label{est2}
|\{ a\in (1,4)^d \ : \ \exists j_1\in R_{\gamma}, \ |g_{j_1}(a)| < \gamma  \}| \lesssim_{r,d} \gamma^{\frac1{r^2}} (\langle j_3 \rangle \dots \langle j_r \rangle)^{2d }.
\end{equation}
\end{lemma}

\begin{proof} Without loss of generality, we assume that $\gamma< \min((2r)^{-2},(36d)^{-1})$. If $j_1\in R_{\gamma}$ recalling that for $x\geq 0$, we have $|\sqrt{1+x}-1|\leq x/2$, we deduce that
$$
|g_{j_1}(a)| \geq    |h_{j_1}(a) | - \frac1{2|j_1|_a^2} -  \frac1{2|j_2|_a^2} \quad \mathrm{where} \quad h_{j_1}(a) := |j_1|_a^2-|j_2|_a^2 + \sigma_1 \sum_{k=3}^r \sigma_k \sqrt{1+|j_k|^4_a}.
$$
However, by definition of $j_2$ and $R_{\gamma}$, we have
$$
|j_2|\geq |j_1|-\sum_{k=3}^r |j_k| \geq \gamma^{-1/2} (\langle j_3 \rangle \dots \langle j_r \rangle)^{2d r^2} - (r-2) (\langle j_3 \rangle \dots \langle j_r \rangle) \geq \frac{\gamma^{-1/2}}2  (\langle j_3 \rangle \dots \langle j_r \rangle)^{2d r^2}.
$$
Noting that, for $a\in (1,4)^d$, we have $|\cdot| \leq |\cdot|_a$, we deduce that
$$
|g_{j_1}(a)| \geq    |h_{j_1}(a) | - 3\gamma (\langle j_3 \rangle \dots \langle j_r \rangle)^{-4d r^2}.
$$
Consequently, it is enough to prove that
\begin{equation}
\label{est3}
|\{ a\in (1,4)^d \ : \ \exists j_1\in R_{1}, \ |h_{j_1}(a)| < \gamma (\langle j_3 \rangle \dots \langle j_r \rangle)^{-4dr^2} \}| \lesssim_{r,d} \gamma^{\frac1{(r-1)(r-2)}}.
\end{equation}
To prove this estimate, we have to note the following result whose proof is postponed to the end of this proof.
\begin{lemma} \label{petit_lemme} If $j_1\in R_{\gamma}$ then there exists $\kappa_{j_1} \in \mathbb{Z}^d$ such that 
$$
|j_1|_a^2 - |j_2|_a^2=\kappa_{j_1}\cdot a, \quad |\kappa_{j_1}|_{\infty} \leq 7(\langle j_3 \rangle \dots \langle j_r \rangle)^3 \quad \mathrm{and} \quad \exists i_{\star}\in \llbracket 1,d\rrbracket, \ \kappa_{j_1,i_{\star}}=0.
$$
\end{lemma}
Now we have to distinguish two cases. 

\medskip

\noindent \underline{$\bullet$\emph{Case $1$: $(\sigma_k,j_k)_{k\geq 3}$ is resonant.}} If $j_1\in R_{\gamma}$, let $\kappa_{j_1}\in \mathbb{Z}^d$ be given by Lemma \ref{petit_lemme}. Note that $\kappa_{j_1}\neq 0$ because else we would have $j_{1,i}^2=j_{2,i}^2$ for all $i\in \llbracket 1,d \rrbracket$ and so $(\sigma,j)$ would be resonant (which is excluded by definition of $R_{\gamma}$).
Furthermore, here $h_{j_1}(a)= \kappa_{j_1} \cdot a$ is a linear form. Consequently, for all $\gamma>0$, we have the following estimate which is much stronger than \eqref{est3}:
\begin{multline*}
|\{ a\in (1,4)^d \ : \ \exists j_1\in R_{1}, \ |h_{j_1}(a)| < \gamma \}| \leq \Big| \bigcup_{\substack{\kappa \in \mathbb{Z}^d \setminus \{0\}\\ |\kappa|_{\infty} \leq 7(\langle j_3 \rangle \dots \langle j_r \rangle)^3}} \{ a\in (1,4)^d \ : \ \kappa \cdot a < \gamma \} \Big| \\
\leq \sum_{\substack{\kappa \in \mathbb{Z}^d \setminus \{0\}\\ |\kappa|_{\infty} \leq 7(\langle j_3 \rangle \dots \langle j_r \rangle)^3}} |\{ a\in (1,4)^d \ : \ \kappa \cdot a < \gamma \}|   \leq \gamma (14(\langle j_3 \rangle \dots \langle j_r \rangle)^3)^d
\end{multline*}

\medskip 

\noindent \underline{$\bullet$\emph{Case $2$: $(\sigma_k,j_k)_{k\geq 3}$ is non-resonant.}} If $j_1\in R_{\gamma}$, $h_{j_1}$ writes
$$
h_{j_1}(a) = \kappa_{j_1} \cdot a + \sum_{k=1}^{\widetilde{r}} n_k \sqrt{1+|\widetilde{j}_k|^4_a}
$$
where $ \kappa_{j_1}$ is given by Lemma \ref{petit_lemme}, $\widetilde{r}\leq r-2$, $(\widetilde{j}_1,\dots,\widetilde{j}_{\widetilde{r}})  \in (\mathbb{N}^d)^{\widetilde{r}}$ is injective, $n_k\in (\mathbb{Z} \setminus \{0\})^d$ is defined by
$$
n_k = \sum_{\substack{i\in \llbracket 3,r \rrbracket \\ \forall \ell,\ |j_{i,\ell}| = \widetilde{j}_{k,\ell} }} \sigma_1 \sigma_i.
$$
Consequently, by Lemma \ref{petit_lemme}, we have
\begin{equation*}
\begin{split}
&|\{ a\in (1,4)^d \ : \ \exists j_1\in R_{1}, \ |h_{j_1}(a)| < \gamma  \}|\\ \leq& \Big| \! \! \! \! \!  \bigcup_{\substack{\kappa \in \mathbb{Z}^d \\  |\kappa|_{\infty} \leq 7(\langle j_3 \rangle \dots \langle j_r \rangle)^3 \\ \exists i_{\star},\ \kappa_{i_{\star}=0}}}  \! \! \! \! \{ a\in (1,4)^d \ : \ |\kappa \cdot a +  \sum_{k=1}^{\widetilde{r}} n_k \sqrt{1+|\widetilde{j}_k|^4_a} | < \gamma \} \Big| \\
\leq& \sum_{\substack{\kappa \in \mathbb{Z}^d \\  |\kappa|_{\infty} \leq 7(\langle j_3 \rangle \dots \langle j_r \rangle)^3 \\ \exists i_{\star},\ \kappa_{i_{\star}=0}}} |\{ a\in (1,4)^d \ : \ |\kappa \cdot a +  \sum_{k=1}^{\widetilde{r}} n_k \sqrt{1+|\widetilde{j}_k|^4_a} | < \gamma \}|.
\end{split}
\end{equation*}
Finally, by applying Lemma \ref{lem_gen} we get
$$
|\{ a\in (1,4)^d \ : \ \exists j_1\in R_{1}, \ |h_{j_1}(a)| < \gamma \}|  \lesssim_{r,d} \gamma^{\frac1{(r-2)(r-1)}}  (\langle j_3 \rangle \dots \langle j_r \rangle)^{\frac{12}{r-1}+3d},
$$
which is also stronger than \eqref{est3}.
\end{proof}
\begin{proof}[\underline{Proof of Lemma \ref{petit_lemme}}] First let us note that
$$
|j_1|_a^2 - |j_2|_a^2 = \kappa_{j_1} \cdot a \quad \mathrm{where} \quad \kappa_{j_1,i} = j_1^2-j_2^2 = \sigma_2 j_{\geq3,i} (j_1+j_2).
$$
First we aim at controlling $|k|_{\infty}$.
If $i\notin I$ then $j_{\geq3,i} =0$ and so $\kappa_{j_1,i}=0$. Else, since $j_1\in \mathbb{Z}^d \setminus \bigcup_{i\in I} C_i$, we have
$|j_{1,i}|\leq 2(1+ \sum_{k\geq 3} |j_{k,i}|^2)$. Consequently, we deduce that
$$
|\kappa_{j_1,i} | \leq \left(\sum_{k\geq 3} |j_{k,i}| \right) \left(4 + 2\sum_{k\geq 3} |j_{k,i}|^2 + \sum_{k\geq 3} |j_{k,i}| \right)  \leq 7(\langle j_3 \rangle \dots \langle j_r \rangle)^3.
$$
Now we assume by contradiction that $\kappa_{j_1,i}\neq 0$ for all $i\in \llbracket 1,d\rrbracket$. Consequently, we have $I=\llbracket 1,d\rrbracket$ and so
\begin{equation}
\label{ohno}
|j_1|_{\infty} \leq 2(1+ \sum_{k\geq 3} |j_{k}|^2) \leq 6 \langle j_3 \rangle \dots \langle j_r \rangle.
\end{equation}
However, since $j_1\in R_{\gamma}$, we have $|j_1| \geq \gamma^{-1/2}(\langle j_3 \rangle \dots \langle j_r \rangle)^{2dr^2}$
which is in contradiction with \eqref{ohno} because we have assumed that $\gamma < (36d)^{-1}$.
\end{proof}

Finally in the following lemma we deal with the general degenerated cases.
\begin{lemma} 
\label{lem_deg_gen}
 For all $\gamma>0$, we have
\begin{equation}
\label{est42}
|\{ a\in (1,4)^d \ : \ \exists j_1\in S, \ |g_{j_1}(a)| < \gamma  \}| \lesssim_{r,d} \gamma^{\frac1{8r^4}} (\langle j_3 \rangle \dots \langle j_r \rangle)^{5d}.
\end{equation}
\end{lemma}
\begin{proof} Without loss of generality we assume that $\gamma\in (0,1)$. Let $\eta\in(0,1)$ be a small number that will be optimized with respect to $\gamma$ later. From the decomposition $S= R_{\eta}\cup (S\setminus R_{\eta})$ we get
\begin{multline}
\label{godzilla}
|\{ a\in (1,4)^d \ : \ \exists j_1\in S, \ |g_{j_1}(a)| < \gamma \}| \leq \sum_{j_1\in S\setminus R_{\eta}} |\{ a\in (1,4)^d \ : \ |g_{j_1}(a)| < \gamma \}| \\ + |\{ a\in (1,4)^d \ : \ \exists j_1\in R_{\eta}, \ |g_{j_1}(a)| < \eta \}| .
\end{multline}
To estimate the sum, we apply Lemma \ref{lem_gen} (with $\kappa=0$) and we get
\begin{equation*}
\begin{aligned}
%\begin{multline*}
\sum_{j_1\in S\setminus R_{\eta}} |\{ a\in (1,4)^d \ : \ |g_{j_1}(a)| < \gamma \}| 
&\leq \! \! \! \! \! \! \! \! \! \! \! \! \! 
\sum_{\substack{|j_1| < \eta^{-1/2}(\langle j_3 \rangle \dots \langle j_r \rangle)^{2dr^2} \\
(\sigma,j) \mathrm{\ is\ non}\tiret\mathrm{resonant} }} \!  \! \! \! \! \! |\{ a\in (1,4)^d \ : \ 
|g_{j_1}(a)| < \gamma \}|
\\&
\leq \! \! \! \! \! \! \! \! \sum_{ |j_1| < \eta^{-1/2}(\langle j_3 \rangle \dots \langle j_r \rangle)^{2dr^2}  } \gamma^{\frac1{r(r+1)}}    (\langle j_1 \rangle \dots \langle j_r \rangle)^{\frac{12}{r+1}}.
%\end{multline*}
\end{aligned}
\end{equation*}
Furthermore, by the zero momentum condition \eqref{eq_zeromom}, since $\eta\in (0,1)$, we also have
$$
|j_2|\lesssim_r \eta^{-1/2}(\langle j_3 \rangle \dots \langle j_r \rangle)^{2dr^2}.
$$
Consequently, we have
\begin{multline*}
\sum_{j_1\in S\setminus R_{\eta}} |\{ a\in (1,4)^d \ : \ |g_{j_1}(a)| < \gamma \}| \lesssim_r \gamma^{\frac1{r(r+1)}}   \eta^{-\frac12- \frac{12}{r+1}} (\langle j_3 \rangle \dots \langle j_r \rangle)^{2dr^2 + \frac{12}{r+1} +  \frac{24}{r+1}2dr^2 }  \\
 \lesssim_r  \gamma^{\frac1{2r^2}} \eta^{-\frac72} (\langle j_3 \rangle \dots \langle j_r \rangle)^{15dr^2 }.
\end{multline*}
Therefore, applying Lemma \ref{lem_deg_large}, we deduce of \eqref{godzilla} that
$$
|\{ a\in (1,4)^d \ : \ \exists j_1\in S, \ |g_{j_1}(a)| < \gamma \}|  \lesssim_{r,d} \eta ^{\frac1{r^2}} (\langle j_3 \rangle \dots \langle j_r \rangle)^{2d }+\gamma^{\frac1{2r^2}} \eta^{-\frac72} (\langle j_3 \rangle \dots \langle j_r \rangle)^{15dr^2 }.
$$
Finally, we get \eqref{est42} by optimizing this last estimate choosing
$$
\eta = \gamma^{\frac1{7r^2+2}} (\langle j_3 \rangle \dots \langle j_r \rangle)^{\frac{15dr^2-2d}{7/2+1/r^2} } .
$$
\end{proof}
\subsection{Proof of Proposition \ref{mainNR}}
\label{sub_proof_mainNR} For $r\geq 3$ let $\mathcal{M}_r$ and $\mathcal{R}_r$ be the sets defined by
$$
\mathcal{M}_r =\{ (\sigma,j)\in (\{-1,1\})^r \times(\mathbb{Z}^d)^{r} \ : \ \sum_{k=1}^r \sigma_k j_k =0\} \quad \mathrm{and}
$$
$$
\mathcal{R}_r = \{ (\sigma,j)\in (\{-1,1\})^r \times(\mathbb{Z}^d)^{r} \ : \ \ (\sigma,j) \mathrm{\ is \ resonant}\}.
$$
On the one hand, as a direct corollary of Lemma \ref{lem_deg_gen} and Corollary \ref{cor_gen_nice}, for all $\gamma>0$ we have
\begin{multline*}
\Big|\{ a \in (1,4)^d  \ : \exists r\geq 3, \exists  (\sigma,j) \in \mathcal{M}_r \setminus \mathcal{R}_r, \ \sigma_1 \sigma_2 = -1 \ \mathrm{and} \\ \big| \sum_{k=1}^r \sigma_k \sqrt{1+|j_k|_a^4} \big| < c_{r,d}\gamma^{8r^4} \langle j_1 \rangle^{-(d-1)} \log^{-2d}(\langle j_1\rangle)  (\langle j_3 \rangle \dots \langle j_r \rangle)^{-44dr^4}\} \Big| < \gamma
\end{multline*}
where $c_{r,d}>0$ is a constant depending only on $r$ and $d$. Consequently, it is enough to prove that for all $\gamma \in (0,1)$, we have
\begin{multline}
\label{what_we_want_to_conclude}
I_{\gamma}:=\Big|\{ a \in (1,4)^d  \ : \exists r\geq 3, \exists  (\sigma,j) \in \mathcal{M}_r \setminus \mathcal{R}_r, \ \sigma_1 \sigma_2 = 1 \ \mathrm{and} \\ \big| \sum_{k=1}^r \sigma_k \sqrt{1+|j_k|_a^4} \big| < \kappa_{r,d} \gamma^{r(r+1)} (\langle j_3 \rangle \dots \langle j_r \rangle)^{-9 d r^2}\} \Big| < \gamma
\end{multline}
where $\kappa_{r,d}\in (0,1)$  is another constant depending only on $r$ and $d$ (and that will be determined later). Indeed, by additivity of the measure, we have
$$
I_{\gamma} \leq \sum_{r\geq 3} \sum_{\substack{(\sigma,j) \in \mathcal{M}_r \setminus \mathcal{R}_r\\ \sigma_1 \sigma_2=1}} \Big|\{ a \in (1,4)^d  \ : \ \big| \sum_{k=1}^r \sigma_k \sqrt{1+|j_k|_a^4} \big| < \kappa_{r,d} \gamma^{r(r+1)} (\langle j_3 \rangle \dots \langle j_r \rangle)^{-9 d r^2}\} \Big|.
$$
Note that if $|j_1| \geq 2\sqrt{r} \langle j_3 \rangle \dots \langle j_r \rangle$ and $\sigma_1 \sigma_2 = 1$ then
\begin{multline*}
\big| \sum_{k=1}^r \sigma_k \sqrt{1+|j_k|_a^4} \big|  \geq \sqrt{1+|j_1|_a^4} - \sum_{k=3}^r \sqrt{1+|j_k|_a^4} \geq \sqrt{1+|j_1|^4} - \sum_{k=3}^r \sqrt{1+16|j_k|^4} \\
\geq \sqrt{1+|j_1|^4} - 4 \sum_{k=3}^r \sqrt{1+|j_k|^4} \geq |j_1|^2 - 4  \sum_{k=3}^r  (1+|j_k|^2) \geq 4 (\langle j_3 \rangle \dots \langle j_r \rangle)^2 >1
\end{multline*}
and so $\Big|\{ a \in (1,4)^d  \ : \ \big| \sum_{k=1}^r \sigma_k \sqrt{1+|j_k|_a^4} \big| < \kappa_{r,d} \gamma^{r(r+1)} (\langle j_3 \rangle \dots \langle j_r \rangle)^{-9d r^2}\} \Big| $ vanishes. Since the same holds if $j_1$ is replaced by $j_2$, 
consequently, we have that $I_{\gamma}$ is bounded from above by 
$$
%I_{\gamma} \leq 
\sum_{r\geq 3} \sum_{\substack{(\sigma,j) \in \mathcal{M}_r \setminus \mathcal{R}_r\\ |j_1| \leq 2\sqrt{r} \langle j_3 \rangle \dots \langle j_r \rangle \\ |j_2| \leq 2\sqrt{r} \langle j_3 \rangle \dots \langle j_r \rangle}} \Big|\{ a \in (1,4)^d  \ : \ \big| \sum_{k=1}^r \sigma_k \sqrt{1+|j_k|_a^4} \big| < \kappa_{r,d} \gamma^{r(r+1)} (\langle j_3 \rangle \dots \langle j_r \rangle)^{-9 d r^2}\} \Big|.
$$
Now denoting by  $c_{r,d}>0$ the constant given by Lemma \ref{lem_gen}, we get
$$
I_{\gamma} \leq   \sum_{r\geq 3} c_{r,d} \sum_{\substack{(\sigma,j) \in \mathcal{M}_r \setminus \mathcal{R}_r\\ |j_1| \leq 2\sqrt{r} \langle j_3 \rangle \dots \langle j_r \rangle \\ |j_2| \leq 2\sqrt{r} \langle j_3 \rangle \dots \langle j_r \rangle}} 
\left( \kappa_{r,d} \gamma^{r(r+1)} (\langle j_3 \rangle \dots \langle j_r \rangle)^{-9 d r^2}  \right)^{\frac1{r(r+1)}}  (\langle j_1 \rangle \dots \langle j_r \rangle)^{\frac{12}{r+1}}.
$$
Consequently, we get an other constant $\widetilde{c}_{r,d}>0$ such that
$$
I_{\gamma} \leq  \gamma  \sum_{r\geq 3} \widetilde{c}_{r,d} \kappa_{r,d}^{\frac1{r(r+1)}} \sum_{j_3,\dots,j_r \in \mathbb{Z}^d} (\langle j_3 \rangle \dots \langle j_r \rangle)^{-9 d \frac{r^2}{r(r+1)} +\frac{36}{r+1}}.
$$
Noting that $9 d \frac{r^2}{r(r+1)} -\frac{36}{r+1}\geq 2d$, we deduce that 
$$
I_{\gamma} \leq \gamma  \sum_{r\geq 3} \widetilde{c}_{r,d}\, \kappa_{r,d}^{\frac1{r(r+1)}} \Big( \sum_{j\in \mathbb{Z}^d} \langle j\rangle^{-2d} \Big)^{r-2}.
$$
Consequently, we deduce a natural choice for $\kappa_{r,d}$ such that $I_{\gamma}<\gamma$ which conclude this proof. 

\section{The  Birkhoff normal form step}
In the rest of the paper we shall fix the parameter $\nu$,
(see \eqref{toriIrr} and \eqref{omegoneBeam})
defining the irrationality of the torus, 
in the full Lebesgue measure set given by 
Proposition \ref{mainNR}. 
For $d\geq2$ and $n\in \mathbb{N}$ we define
\begin{equation}\label{def:Mnd}
M_{d,n}:=\left\{
\begin{aligned}
&n+2(n-2)+1 \quad\quad\; {\rm if}\;\; d=2\;\; {\rm and}\;\; n\;\; {\rm odd}\\
&n+2(n-2) \qquad\quad\quad {\rm if}\;\; d=2\;\; {\rm and}\;\; n\;\; {\rm even}\\
&n+\phantom{2}(n-2) \qquad\quad\quad {\rm if}\;\; d=3\\
&n\phantom{22}
\;\,\qquad\qquad\qquad\quad\;\, {\rm if} \;\; d\geq4\,.
\end{aligned}\right.
\end{equation}
%%%%%%%%THIS IS ANOTHER WAY TO WRITE M%%%%%%%%%
%\begin{equation}\label{def:Mnd}
%M_{d,n}:=\left\{
%\begin{aligned}
%&3n-3\quad {\rm if}\;\; d=2\;\; {\rm and}\;\; n\;\; {\rm odd}\\
%&3n-4\quad {\rm if}\;\; d=2\;\; {\rm and}\;\; n\;\; {\rm even}\\
%&2n-2\quad {\rm if}\;\; d=3\\
%&\phantom{3}n\phantom{-2}
%\;\,\quad {\rm if} \;\; d\geq4
%\end{aligned}\right.
%\end{equation}
The main result of this section is the following.

\begin{theorem}\label{BNFstep} 
Let $d=2,3$ and let $r\in \mathbb{N}$ such that  $M_{d,n}\leq r\leq 4n$.
 There exits $\beta=\beta(d,r)>0$ 
 %depending only of the exponent $\alpha$ appearing in 
 %Proposition \ref{mainNR} 
 such that for any 
$N\geq 1$, any $\delta>0$ 
and $s\geq s_0=s_0(\beta)$, 
there exist $\eps_0\lesssim_{s,\delta} N^{-\delta}$ 
and two canonical transformation $\tau^{(0)}$ and $\tau^{(1)}$ 
making the following diagram to commute
\begin{equation}
\label{diagram}
\xymatrixcolsep{5pc} \xymatrix{  B_s(0,\eps_0) \ar[r]^{ \tau^{(0)} }
 \ar@/_1pc/[rr]_{\mathrm{id}_{{H}^s}} & B_s(0,2\eps_0)  \ar[r]^{ \tau^{(1)} }  & 
 {H}^s(\mathbb{T}^{d})  } 
\end{equation}
and close to the identity
\begin{align}\label{tau}
\forall \sigma\in \{0,1\}, \ \|u\|_{{H}^s} <2^{\sigma}\eps_0 \;\; \Rightarrow \;\;  
\|\tau^{(\sigma)}(u)-u\|_{{H}^s} \lesssim_{s,\delta} N^{\delta} \|u\|_{{H}^s}^2
\end{align}
such that, on $B_s(0,2\eps_0)$, $H \circ \tau^{(1)}$ writes
\begin{equation}\label{HamTau1}
H \circ \tau^{(1)} = Z_2+ 
\sum_{k=n}^{M_{d,n}-1}Z_{k}^{\leq N}+\sum_{k= M_{d,n}}^{r-1}K_{k}+K^{>N}+\tilde{R}_{r}
\end{equation}
where $M_{d,n}$ is given in \eqref{def:Mnd} and where

\vspace{0.4em}
\noindent
$(i)$
 $Z_{k}^{\leq N}$, for $k=n,\ldots, M_{d,n}-1$,
are resonant Hamiltonians of order $k$
given by the formula
\begin{equation}\label{Z4}
Z_k^{\leq N} = \sum_{\substack{\sigma\in\{-1,1\}^k,\ j\in(\Z^{d})^{k},\ \mu_2(j)\leq N\\
\sum_{i=1}^k\sigma_i j_i=0\\
\sum_{i=1}^k\sigma_i \omega_{j_i}=0}}
(Z_{k}^{\leq N})_{\sigma,j}u_{j_1}^{\sigma_1}\cdots u_{j_k}^{\sigma_k}\,, 
\quad |(Z_{k}^{\leq N})_{\sigma,j}|\lesssim_{\delta}N^{\delta} \frac{\mu_3(j)^\beta}{\mu_1(j)}\,;
\end{equation}

\vspace{0.3em}
\noindent
$(ii)$
 $K_k$, $k=M_{d,n},\ldots, r-1$,  
 are homogeneous polynomials of order $k$
 \begin{equation}\label{K6}
K_k= \sum_{\substack{\sigma\in\{-1,1\}^k,\ j\in(\Z^{d})^{k}\\
\sum_{i=1}^k\sigma_i j_i=0}}
(K_{k})_{\sigma,j}u_{j_1}^{\sigma_1}\cdots u_{j_k}^{\sigma_k}\,,
\quad |(K_k)_{\sigma,j}|\lesssim_{\delta} N^{\delta} \mu_3(j)^\beta\,;
\end{equation}

\vspace{0.3em}
\noindent
$(iii)$
$K^{>N}$ and $\tilde{R}_{r}$ are remainders 
 satisfying
 \begin{align}
 \|X_{K^{>N}}(u)\|_{H^{s}}&\lesssim_{s,\delta}N^{-1+\delta}\|u\|_{H^{s}}^{n-1}\,,\label{KmagNN}\\
\| X_{\tilde R_r}(u)\|_{H^{s}}&\lesssim_{s,\delta} N^{\delta} \| u\|_{H^s}^{r-1}\,.\label{R8tilde}
\end{align}
\end{theorem}

 It is convenient to introduce the following class.
\begin{definition}{\bf (Formal Hamiltonians)}\label{Ham:class}
Let $N\in \mathbb{R}$, $k\in \mathbb{N}$
with $k\geq3$ and $N\geq1$.
%Let $k\geq 3$ be a natural number and $q\geq 0$ real. 

\noindent
$(i)$ We denote by $ \mathcal{L}_k$ 
the set of Hamiltonian  having homogeneity 
$k$ and such that they may be written in the form
\begin{align}
G_{k}(u)&= 
\sum_{\substack{\sigma_i\in\{-1,1\},\ j_i\in\Z^d\\
\sum_{i=1}^k\sigma_i j_i=0}}
(G_{k})_{\sigma,j}u_{j_1}^{\sigma_1}\cdots u_{j_k}^{\sigma_k}\,,
\quad 
(G_{k})_{\sigma,j}\in \mathbb{C}\,, 
\quad \begin{array}{cl}&\sigma:=(\sigma_1,\ldots,\sigma_k)\\
&j:=(j_1,\ldots,j_k)\end{array}\label{HamG} 
\end{align}
 with symmetric coefficients $(G_k)_{\s,j}$, i.e.
%  satisfy the estimate
%\begin{equation}\label{coeffGp}
%|(G_{k})_{\sigma,j}|\lesssim 
%N^{\zeta}\mu_{3}(j)^{\beta}\mu_{1}(j)^{-q}\,,\qquad 
%%\forall \s\in\{-1,+1\}^{k}\,,\; j\in \mathbb{Z}^{kd}\,,
%\end{equation}
%for a certain $\beta>0$ and are symmetric, i.e. 
for any $\rho\in\mathfrak{S}_{k}$
one has $(G_{k})_{\sigma,j}=(G_{k})_{\sigma\circ\rho,j\circ\rho}$. 

\noindent
$(ii)$ If $G_{k}\in \mathcal{L}_{k}$ then $G_{k}^{>N}$ denotes the element of 
$\mathcal{L}_{k}$ defined by
\begin{equation}\label{HamGhigh}
(G_{k}^{>N})_{\sigma,j}:=\left\{
\begin{array}{lll}
&(G_{k})_{\sigma,j}\,,&{\rm if} \;\;\mu_{2}(j)>N\,,\\
&0\,,  & {\rm else}\,.
\end{array}
\right.
\end{equation}
%
%We denote by $G_{k}^{>N}(u)$ an element in $\mathcal{L}_k$ 
%defined as
%%which is supported on the modes $\mu_2(j)>N$, more precisely
%\begin{align}
%G_{k}^{>N}(u)&=
%\sum_{\substack{\sigma_i\in\{-1,1\},\ j_i\in\Z^d\\
%\sum_{i=1}^k\sigma_i j_i=0,\, \mu_{2}(j)>N}}
%(G_{k})_{\sigma,j}u_{j_1}^{\sigma_1}\cdots u_{j_k}^{\sigma_k}\,,\label{HamGhigh}
%\end{align}
We set
$G^{\leq N}_{k}:=G_{k}-G^{>N}_{k}$. 
\end{definition}

\begin{remark}
Consider the Hamiltonian $H$ in \eqref{beamHam} and its Taylor expansion in 
\eqref{expinitial}. One can note that the Hamiltonians $H_{k}$ in \eqref{H8}
belong to the class $\mathcal{L}_{k}$.
This follows form the
% bound \eqref{initCoeff} on the coefficients 
% and from  the 
 fact that, without loss of generality, one can substitute 
 the Hamiltonian $H_{k}$ with its symmetrization.
\end{remark}

We also need the following definition.
\begin{definition}\label{def:adZ2}
Consider the Hamiltonian $Z_2$ in \eqref{quadratHam} 
and
$G_{k}\in \mathcal{L}_{k}$.
%of the form
%\begin{equation}\label{HamGGPP}
%G_{p}(u)= 
%\sum_{\substack{\sigma\in\{-1,1\}^p,\ j\in(\Z^{d})^{p}\\
%\sum_{i=1}^p\sigma_i j_i=0}}
%(G_{p})_{\sigma,j}u_{j_1}^{\sigma_1}\cdots u_{j_p}^{\sigma_p}\,,
%\qquad  (G_{p})_{\sigma,j}\in \mathbb{C}\,,
%\end{equation}
%with $p\in \mathbb{N}$, $p\geq3$. 

\noindent
$\bullet$ {\bf (Adjoint action).}
We define the adjoint action ${\rm ad}_{Z_2}G_k$ in $\mathcal{L}_{k}$
by 
\begin{equation}\label{adZ2}
({\rm ad}_{Z_2}G_k)_{\s,j}:=
\Big(\ii \sum_{i=1}^{k}\s_i\omega_{j_i}\Big)
(G_{k})_{\sigma,j}\,.
\end{equation}

\noindent
$\bullet$ {\bf (Resonant Hamiltonian).}
We define $G_{k}^{\rm res}\in \mathcal{L}_{j}$ by 
\[
(G_{k}^{res})_{\s,j}:=(G_{k})_{\s,j}\,,\;\;{\rm when}\;\;\; 
\sum_{i=1}^{k}\s_i\omega_{j_i}=0
\]
and $(G_{k}^{\rm res})_{\s,j}=0$ otherwise.

\noindent
$\bullet$ We define $G_{k}^{(+1)}\in \mathcal{L}_{k}$ 
by
\[
\begin{aligned}
&(G_{k}^{(+1)})_{\s,j}:=(G_{k})_{\s,j}\,,\;\;{\rm when}\;\;\; 
\exists i,p=1,\ldots,k\;{\rm s.t.}\;\\
& \mu_{1}(j)=|j_{i}|\,,\;\; \mu_{2}(j)=|j_{p}|
 \;\;{\rm and}\;\; \s_{i}\s_{p}=+1\,.
\end{aligned}
\]
We define $G_{k}^{(-1)}:=G_k-G_{k}^{(+1)}$.
\end{definition}

%\begin{definition}\label{def:splitting}
%Consider an Hamiltonian $G_p$ of the form \eqref{HamGGPP}.
%
%\vspace{0.2em}
%\noindent
%$\bullet$
%We define  $G_{p}^{\leq N}$ as the Hamiltonian  of the form \eqref{HamGGPP}
%with coefficients $(G_{p}^{\leq N})_{\s,j}$ defined as
%\[
%(G_{p}^{\leq N})_{\s,j}:=\left\{
%\begin{aligned}
%&(G_{p})_{\s,j}\;\;\;{\rm if }\;\;\; \mu_2(j)\leq N\,,\\
%&0 \qquad \;\;\;\;\; {\rm otherwise}\,,
%\end{aligned}\right.
%\]
%%\[
%%(G_{p}^{\leq N})_{\s,j}:=(G_{p})_{\s,j}\,,\;\;{\rm if }\;\;\; \mu_2(j)\leq N\,,\quad 
%%(G_{p}^{\leq N})_{\s,j}:=0\,,\;\; {\rm otherwise}\,,
%%\]
%for $\s\in \{-1,+1\}^{p}$ and $j\in(\Z^{d})^{p}$;
%we define $G_{p}^{> N}:=G_{p}-G_{p}^{\leq N}$.
%
%\vspace{0.2em}
%\noindent
%$\bullet$ We define $G_{p}^{(+1)}$ as the Hamiltonian  
%of the form \eqref{HamGGPP}
%with coefficients $(G_{p}^{(+1)})_{\s,j}$,  
%$\s\in \{-1,+1\}^{p}$ and $j\in(\Z^{d})^{p}$,  defined as
%\[
%\begin{aligned}
%&(G_{p}^{(+1)})_{\s,j}:=(G_{p})_{\s,j}\,,\;\;{\rm when}\;\;\; 
%\exists i,k=1,\ldots,p\;{\rm s.t.}\;\\
%& \mu_{1}(j)=|j_{i}|\,,\;\; \mu_{2}(j)=|j_{k}|
% \;\;{\rm and}\;\; \s_{i}\s_{k}=+1\,.
%\end{aligned}
%\]
%We define $G_{p}^{(-1)}:=G_p-G_{p}^{(+1)}$.
%%and $(G_{p}^{(\eta)})_{\s,j}=0$ otherwise.
%
%\vspace{0.2em}
%\noindent
%$\bullet$ We define $G_{p}^{\rm res}$ as the Hamiltonian  of the form \eqref{HamGGPP}
%with coefficients $(G_{p}^{\rm res})_{\s,j}$,  $\s\in \{-1,+1\}^{p}$ and $j\in(\Z^{d})^{p}$, 
%defined as
%\[
%(G_{p}^{res})_{\s,j}:=(G_{p})_{\s,j}\,,\;\;{\rm when}\;\;\; 
%\sum_{i=1}^{p}\s_i\omega_{j_i}=0
%\]
%and $(G_{p}^{\rm res})_{\s,j}=0$ otherwise.
%\end{definition}

\begin{remark}\label{Rmk:resonant}
Notice that, in view of Proposition \ref{mainNR}, the resonant Hamiltonians given in  
Definition \ref{def:adZ2} must be supported on indices $\s\in\{-1,1\}^{k}$, 
$j\in\mathbb{Z}^{kd}$ which are  resonant according to Definition \ref{def_res}.
We remark that $(G_{k})^{\rm res}\equiv0$ if $k$ is odd.
\end{remark}

In the following lemma we collect  some properties of the Hamiltonians 
in Definition \ref{Ham:class}.

\begin{lemma}\label{lem:hamvec}
Let $N\geq1$, $0\leq \delta_i<1$, $q_i\in\R$, $k_i\geq3$, 
consider 
$G^i_{k_i}(u)$ in $\mathcal{L}_{k_i}$ for $i=1,2$. Assume
%. We have the following.
%and consider $G_k$ in \eqref{HamG}. Assume 
that the coefficients 
$(G^i_{k_i})_{\sigma,j}$ satisfy
\begin{equation}\label{coeffGp}
|(G^i_{k_i})_{\sigma,j}|\leq C_{i} N^{\delta_i}
\mu_{3}(j)^{\beta_i}\mu_{1}(j)^{-q_{i}}\,,
\qquad \forall \s\in\{-1,+1\}^{k}\,,\; j\in \mathbb{Z}^{kd}\,,
\end{equation}
for some $\beta_i>0$ and $C_i>0$, $i=1,2$.

\noindent {\bf (i) (Estimates on  Sobolev spaces)} 
Set $k=k_i$, $\delta=\delta_i$, $q=q_i$, $\beta=\beta_i$, $C=C_i$ 
and $G^i_{k_i}=G_k$ for $i=1,2$. There is $s_0=s_0(\beta,d)$ such that 
 for $s\geq s_0$,  $G_{k}$ defines naturally a smooth function
 from $H^{s}(\mathbb{T}^{d})$ to $\mathbb{R}$. In particular
 one has the following estimates: 
\begin{align}
|G_{k}(u)|
&\lesssim_{s}CN^{\delta}\|u\|_{H^{s}}^{k}\label{est:energy-form},
\\
\|X_{G_k}(u)\|_{H^{s+q}}
&\lesssim_s CN^{\delta}\|u\|_{H^{s}}^{k-1}\,,\label{est:vector-field}
\\
\|X_{G_{k}^{>N}}(u)\|_{H^{s}}
&\lesssim_{s} CN^{-q+\delta}\|u\|^{k-1}_{H^{s}}\,, \label{est:smooth-vec-fie}
\end{align}
for any $u\in H^{s}(\mathbb{T}^{d}).$ \\
%\mathbb{C})$, where $X_{H}(u)=\ii J \nabla H(u)$
%is the Hamiltonian vector field of $H\in \{G_k(u),\,G_k^{> N}(u)\}$.\\
\noindent{\bf (ii) (Poisson bracket)} The Poisson bracket between $G^1_{k_1}$ and $G^{2}_{k_2}$ is an element of $\mathcal{L}_{k_1+k_2-2}$ and it verifies the estimate
\begin{equation}\label{stimaPoisbra}
|(\{G^1_{k_1},G^2_{k_2}\})_{\sigma,j}|\lesssim_{s} C_1 C_2 N^{\delta_1+\delta_2}\mu_{3}^{\beta_1+\beta_2}\mu_1(j)^{-\min\{q_1,q_2\}},
\end{equation}
for any $\sigma\in\{+1,-1\}^{k_1+k_2-1}$ and $ j\in\mathbb{Z}^{d(k_1+k_2-2)}.$
\end{lemma}

\begin{proof}
We prove item {\bf (i)}. 
Concerning the proof of \eqref{est:energy-form} 
it is sufficient to give the proof in the case $q=0$. 
For convenience, without loss of generality, we assume $C_{i}=1$, $i=1,2$.
We have 
\begin{align*}
|G_k(u)|&\leq k!\sum_{\substack{j_1,\ldots,j_k\in\Z^d \\ 
|j_1|\geq |j_{2}|\geq|j_3|\geq \ldots\geq |j_{k}|
}} 
|(G_k)_{\sigma,k}||u_{j_1}^{\sigma_1}|\cdots|u_{j_k}^{\sigma_k}|
\\&
\lesssim_{k} N^{\delta}\sum_{j_3\in\Z^d}|j_3|^{\beta}|u_{j_3}^{\sigma_3}| 
\prod_{3 \neq i=1}^k\sum_{j_i\in\Z^d}|u_{j_i}^{\sigma_i}|
%\\&
\lesssim_{k,\epsilon} N^{\delta}\|u\|_{H^{d/2+\beta+\epsilon}}\|u\|_{H^{d/2+\epsilon}}^{k-1},
\end{align*}
for any $\epsilon>0$, we proved the \eqref{est:energy-form} 
with $s_0=d/2+\epsilon+\beta$. 
 
\noindent 
We now prove \eqref{est:vector-field}. Since the coefficients of $G_k$ are symmetric, we have
$$
\partial_{\bar{u}_n}G_k(u) = k \sum_{\sigma_1 j_1+\dots+\sigma_{k-1} j_{k-1}=n} (G_k)_{(\sigma,-1),(j,n)} u_{j_1}^{\sigma_1} \dots u_{j_{r-1}}^{\sigma_{r-1}}
$$
 Therefore, we have
\begin{equation*}
\begin{split}
\langle n\rangle^{s+q}|\partial_{\bar{u}_n}G_k(u)| &\leq k !  \! \! \! \! \! \! \! \! \! \! \! \! \sum_{\substack{\sigma_1 j_1+\dots+\sigma_{k-1} j_{k-1}=n \\ |j_1|\geq \dots \geq |j_{k-1}|}}  \! \! \! \! \! \!  \! \! \! \! \! \!  |(G_k)_{(\sigma,-1),(j,n)}| |u_{j_1}^{\sigma_1}| \dots |u_{j_{r-1}}^{\sigma_{r-1}}| \langle n\rangle^{s+q} \\
&\mathop{\lesssim}^{\eqref{coeffGp}}  N^{\delta} \! \! \! \! \! \! \! \! \! \! \! \! \sum_{\substack{\sigma_1 j_1+\dots+\sigma_{k-1} j_{k-1}=n \\ |j_1|\geq \dots \geq |j_{k-1}|}}  \! \! \! \! \! \!  \! \! \! \! \! \! \mu_{3}(j,n)^{\beta}\mu_{1}(j,n)^{-q}  |u_{j_1}^{\sigma_1}| \dots |u_{j_{r-1}}^{\sigma_{r-1}}| \langle n\rangle^{s+q}.
\end{split}
\end{equation*}
We note that in the last sum above, we have $\langle n\rangle \lesssim \langle j_1 \rangle$, $\mu_{1}(j,n)\geq \langle j_1 \rangle$ and $\mu_{3}(j,n)\leq \langle j_2 \rangle$. As a consequence, we deduce that
\begin{equation*}
\begin{split}
\langle n\rangle^{s+q}|\partial_{\bar{u}_n}G_k(u)|  &\lesssim_{s} N^{\delta} \! \! \! \! \! \! \! \! \! \! \! \! \sum_{\substack{\sigma_1 j_1+\dots+\sigma_{k-1} j_{k-1}=n \\ |j_1|\geq \dots \geq |j_{k-1}|}}  \! \! \! \! \! \!  \! \! \! \! \! \!  \langle j_1 \rangle^{s} \langle j_2 \rangle^{\beta}  |u_{j_1}^{\sigma_1}| \dots |u_{j_{r-1}}^{\sigma_{r-1}}| \\ &\lesssim_{s} N^{\delta}  \! \! \! \! \! \! \! \! \sum_{j_1+\dots+j_{k-1}=n }  \! \!  \! \! \! \! \! \!  \langle j_1 \rangle^{s} \langle j_2 \rangle^{\beta}  |u_{j_1}| \dots |u_{j_{r-1}}|.
\end{split}
\end{equation*}
Consequently, applying the Young convolutional inequality, we get
\begin{multline*}
\| X_{G_k}(u) \|_{H^{s+q}} = \|(\langle n\rangle^{s+q}|\partial_{\bar{u}_n}G_k(u)|)_{n\in \mathbb{Z}^d} \|_{\ell^2} \lesssim_s N^{\delta}\| u \|_{H^s} \big( \sum_{j\in \mathbb{Z}^d} \langle j \rangle^\beta |u_j| \big)  \big( \sum_{j\in \mathbb{Z}^d}  |u_j| \big)^{k-3} \\
 \lesssim_s N^{\delta} \|u\|_{H^s}^{k-1}.
\end{multline*}

The proof of \eqref{est:smooth-vec-fie} follows the same lines. 
The proof of item {\bf (ii)} of the lemma 
is a direct consequence of the previous computations, 
definition \eqref{Poissonbrackets} and the momentum condition.
\end{proof}

We are in position to prove the main Birkhoff result.

\vspace{0.3em}
\noindent
{\bf Proof of Theorem \ref{BNFstep}.}
In the case $d=2$
we perform two steps of Birkhoff normal form procedure, 
see Lemmata
\ref{coniugo1}, \ref{coniugo2}.
The case $d=3$ is slightly different. Indeed, due to the 
estimates on the small divisors given in Proposition \ref{mainNR}, 
we can note  that the Hamiltonian in \eqref{newHam1} has already 
the form \eqref{HamTau1} since the coefficients of the Hamiltonians 
$\tilde{K}_{k}$ (see \eqref{coeffTildeK}) do not decay anymore 
in the largest index $\mu_1(j)$.
The proof of Theorem \ref{BNFstep} is then concluded 
after just one step of Birkhoff normal form.

\vspace{0.5em}
\noindent
{\bf Step 1 if $d=2$ or $d=3$.}
We have the following Lemma.
\begin{lemma}{\bf (Homological equation 1).}\label{lem:homoeq}
%Let $N,\delta$ as in Theorem \ref{BNFstep} and
Let $q_{d}=3-d$ for $d=2,3$. For  any $N\geq1$ and $\delta>0$
there exist multilinear Hamiltonians $\chi^{(1)}_{k}$, $k=1,\ldots , 2n-3$
in the class $\mathcal{L}_{k}$ with  coefficients 
$(\chi_{k}^{(1)})_{\s,j}$ satisfying
%of $\chi^{(1)}_{k}$ satisfies 
%of the form
%\[
%\chi^{(1)}_{k}(u)= 
%\sum_{\substack{\sigma\in\{-1,1\}^k,\ j\in(\Z^{d})^k\\
%\sum_{i=1}^k\sigma_i j_i=0, \mu_2(j)\leq N}}
%(\chi^{(1)}_{k})_{\sigma,j}u_{j_1}^{\sigma_1}\cdots u_{j_k}^{\sigma_k}\,,\quad
%(\chi^{(1)}_{k})_{\sigma,j}\in \mathbb{C}
%\]
%with coefficients satisfying 
\begin{equation}\label{stimacoeffchik}
|(\chi_{k}^{(1)})_{\s,j}|\lesssim_{\delta}
N^{\delta} \mu_3(j)^{\beta}\mu_1(j)^{-q_{d}}\,,
%\quad q_{d}:=\left\{
%\begin{aligned}
%&q_d=1\;\;{\rm for}\; d=2\\
%&q_d=0\;\;{\rm for}\; d=3\,;
%\end{aligned}\right.
\end{equation}
such that (recall Def. \ref{def:adZ2})
\begin{equation}\label{homoeq1}
\{\chi_{k}^{(1)},Z_{2}\}+H_{k}=Z_{k}+H_{k}^{>N}\,,\quad k=n,\ldots,2n-3\,,
\end{equation}
where $Z_2$, $H_k$ are given in \eqref{quadratHam}, \eqref{H8}
and $Z_{k}$ is the resonant Hamiltonian defined as
\begin{equation}\label{resonantHamZZ}
Z_{k}:=(H_{k}^{\leq N})^{\rm res}\,, \quad k=n,\ldots,2n-3\,.
\end{equation}
Moreover $Z_{k}$ belongs to
$\mathcal{L}_{k}$ and  has  coefficients  satisfying \eqref{Z4}.
\end{lemma}
\begin{proof}
Consider the Hamiltonians $H_{k}$ in \eqref{H8} with coefficients 
satisfying \eqref{initCoeff}.
Recalling Definition \ref{Ham:class} we write
\begin{equation*}
H_{k}=Z_{k}+(H_{k}^{\leq N}-Z_{k})+H^{>N}_{k}\,,\quad k=n,\ldots,r-1\,,
\end{equation*}
with $Z_k$ as in \eqref{resonantHamZZ}. 
We define
\begin{equation}\label{resonantHamZZ2}
\chi_{k}^{(1)}:=({\rm ad }_{Z_2})^{-1}\big[H_{k}^{\leq N}-Z_{k}\big]\,,\quad k=n,\ldots,2n-3\,,
\end{equation}
 where ${\rm ad}_{Z_2}$ is given by Definition
\ref{def:adZ2}.
%In view of Proposition \ref{mainNR} the Hamiltonians $\chi_{k}^{(1)}$ are well-defined.
In particular (recall formula \eqref{adZ2}) their coefficients 
have the form
\begin{equation}\label{coeffchik}
(\chi_{k}^{(1)})_{\s,j}:=(H_{k})_{\s,j}\Big(\ii \sum_{i=1}^{k}\s_i\omega_{j_i}\Big)^{-1}
\end{equation}
for indices $\s\in\{-1,+1\}^{k}$, $j\in(\Z^{d})^{k}$ such that
\[ 
\sum_{i=1}^{k}\s_i j_i=0\,, \;\;\;\mu_{2}(j)\leq N\;\;\;{\rm and}\;\;\;
\sum_{i=1}^{k}\s_i\omega_{j_i}\neq 0\,.
\]
By \eqref{initCoeff} and Proposition \ref{mainNR} (with $d=2,3$) 
we deduce
the bound \eqref{stimacoeffchik}
for some $\beta>0$. 
The resonant Hamiltonians $Z_{k}$ in \eqref{resonantHamZZ}
have the form \eqref{Z4}.
One can check by an explicit computation that equation \eqref{homoeq1}
is verified.
\end{proof}
We shall use the Hamiltonians $\chi^{(1)}_{k}$ given by Lemma 
\ref{lem:homoeq} to generate a symplectic change of coordinates.
\begin{lemma}\label{lem:flowmap}
Let us define
\begin{equation}\label{hamchi111}
\chi^{(1)}:=\sum_{k=n}^{2n-3}\chi^{(1)}_{k}\,.
\end{equation}
There is $s_0=s_0(d,r)$ such that for any $\delta>0$, for any $N\geq1$ and 
any $s\geq s_0$, if
$\eps_0\lesssim_{s,\delta} N^{-\delta}$, 
then
the problem
\begin{equation}\label{accendino}
\left\{
\begin{aligned}
&\pa_{\tau}Z(\tau)=X_{{\chi^{(1)}}}(Z(\tau))\\
&Z(0)=U=\vect{u}{\bar{u}}\,,\quad u\in B_{s}(0,\eps_0)
\end{aligned}
\right.
\end{equation}
has a unique solution $Z(\tau)=\Phi^{\tau}_{\chi^{(1)}}(u)$
belonging to $C^{k}([-1,1];H^{s}(\mathbb{T}^{d}))$ for any $k\in\mathbb{N}$.
Moreover the map 
$\Phi_{\chi^{(1)}}^{\tau} : B_{s}(0,\eps_0)\to H^{s}(\mathbb{T}^{d})$
is symplectic.
% w.r.t. the symplectic form \eqref{dueforma}.
The flow map $\Phi^{\tau}_{\chi^{(1)}}$ and its inverse 
$\Phi^{-\tau}_{\chi^{(1)}}$ satisfy
\begin{equation*}%\label{stimaflussoStep1}
\begin{aligned}
&\sup_{\tau\in[0,1]}
\|\Phi^{\pm\tau}_{\chi^{(1)}}(u)-u\|_{H^{s}}
\lesssim_{s,\delta}  N^{\delta}\|u\|_{H^{s}}^{n-1}\,,
\\&
\sup_{\tau\in[0,1]}
\|{\rm d}\Phi^{\pm\tau}_{\chi^{(1)}}(u)[\cdot]\|_{\mathcal{L}(H^{s};H^{s})}
\leq 2\,.
\end{aligned}
\end{equation*}
\end{lemma}

\begin{proof}
By estimate \eqref{stimacoeffchik} and Lemma \ref{lem:hamvec}
we have that the vector field $X_{\chi^{(1)}}$
is a bounded operator on $H^{s}(\mathbb{T}^{d})$. Hence the flow
$\Phi^{\tau}_{\chi^{(1)}}$ is well-posed by standard 
theory of Banach space
ODE. The estimates of the map and its differential follow by using 
the equation in \eqref{accendino}, the fact that $\chi^{(1)}$ is multilinear 
and the smallness condition on $\eps_0$.
Finally the map is symplectic since 
it is generated by a Hamiltonian vector field.
\end{proof}

We now study how changes the Hamiltonian $H$ in \eqref{expinitial}
under the map $\Phi^{\tau}_{\chi^{(1)}}$.
\begin{lemma}{\bf (The new Hamiltonian 1).}\label{coniugo1}
There is $s_0=s_0(d,r)$ such that for any $N\geq1$, $\delta>0$
and any $s\geq s_0$, if 
$\eps_0\lesssim_{s,\delta} N^{-\delta}$ then
we have that
\begin{equation}\label{newHam1}
H\circ\Phi_{\chi^{(1)}}=Z_{2}+\sum_{k=n}^{2n-3}Z_{k}+\widetilde{K}^{>N}
+\sum_{k=2n-2}^{r-1}\widetilde{K}_{k}+\mathcal{R}_{r}
\end{equation}
where

\noindent
$\bullet$ $\Phi_{\chi^{(1)}}:=(\Phi^{\tau}_{\chi^{(1)}})_{|\tau=1}$ is the flow 
map given by Lemma \ref{lem:flowmap};

\noindent
$\bullet$ the resonant Hamiltonians $Z_k$ are defined in \eqref{resonantHamZZ};

\noindent
$\bullet$ $\widetilde{K}_k$ are in $\mathcal{L}_{k}$ 
 with coefficients $(\widetilde{K}_k)_{\s,j}$ satisfying 
\begin{equation}\label{coeffTildeK}
|(\widetilde{K}_k)_{\s,j}|\lesssim_{\delta} 
N^{\delta} \mu_3(j)^{\beta}\mu_{1}(j)^{-q_{d}}\,,
\quad k=2n-2,\ldots, r-1\,,
\end{equation}
with $q_{d}=3-d$ for $d=2,3$;
% defined in \eqref{stimacoeffchik};

\noindent
$\bullet$ the Hamiltonian $\widetilde{K}^{>N}$ and 
 the remainder $\mathcal{R}_r$ satisfy
 \begin{align}
\|X_{\widetilde{K}^{>N}}(u)\|_{H^{s}}&\lesssim_{s,\delta}
N^{-1}\|u\|_{H^{s}}^{2}\,,\label{KmagNN22}
\\
\|X_{\mathcal{R}_r}(u)\|_{H^{s}}&\lesssim_{s,\delta} N^{\delta}\|u\|^{r-1}_{H^{s}}\,,
\qquad \forall u\in B_{s}(0,2\eps_0)\,.\label{stimacalRR} 
\end{align}
\end{lemma}
\begin{proof}
Fix $\delta>0$ and $\eps_0N^{\delta}$ small enough. We apply Lemma \ref{lem:flowmap} with $\delta\rightsquigarrow\delta'$
to be chosen small enough with respect to $\delta$ we have fixed
(which ensures us that the smallness condition $\eps_0N^{\delta'}\lesssim_{s,\delta'}1$ 
of Lemma \ref{lem:flowmap} is fulfilled).
Let  $\Phi^{\tau}_{\chi^{(1)}}$ be the flow at time $\tau$ of the Hamiltonian 
$\chi^{(1)}$.% given in Lemma \ref{lem:flowmap}.
We note that
\[
\pa_{\tau}H\circ\Phi^{\tau}_{\chi^{(1)}}=dH(z)[X_{\chi^{(1)}}(z)]_{|z=\Phi^{\tau}_{\chi^{(1)}}}
\stackrel{\eqref{diffHam}, \eqref{Poissonbrackets}}{=}\{\chi^{(1)}, H\}\circ\Phi^{\tau}_{\chi^{(1)}}\,.
\]
Then, for $L\geq2$, we get the Lie series expansion
\[
H\circ\Phi_{\chi^{(1)}}=H+\{\chi^{(1)}, H\}
+\sum_{p=2}^{L}\frac{1}{p!}{\rm ad}_{\chi^{(1)}}^{p}\big[H\big]
+\frac{1}{L!}\int_{0}^{1}(1-\tau)^{L}
{\rm ad}_{\chi^{(1)}}^{L+1}\big[H\big]\circ\Phi^{\tau}_{\chi^{(1)}}\mathrm{d}\tau
\]
where ${\rm ad}_{\chi^{(1)}}^{p}$ is defined recursively as
\begin{equation}\label{adjointiter}
{\rm ad}_{\chi^{(1)}}[H]:=\{\chi^{(1)},H\}\,,\quad 
{\rm ad}_{\chi^{(1)}}^{p}[H]:=
\big\{\chi^{(1)}, {\rm ad}_{\chi^{(1)}}^{p-1}[H] \big\}\,,\;\;p\geq 2\,.
\end{equation}
Recalling the Taylor expansion of the Hamiltonian $H$ in  \eqref{expinitial} we obtain
\begin{align}
H\circ\Phi_{\chi^{(1)}}
&=
Z_{2}+\sum_{k=n}^{2n-3}\Big(H_{k}
+\{\chi_{k}^{(1)},Z_2\}\Big)+\sum_{k=2n-2}^{r-1}H_{k}\nonumber
\\&
+\sum_{p=2}^{L}\frac{1}{p!}{\rm ad}^{p}_{\chi^{(1)}}[Z_2]
+\sum_{j=n}^{r-1}\sum_{p=1}^{L}
\frac{1}{p!}{\rm ad}_{\chi^{(1)}}^{p}[H_{j}]\label{eq:34}
\\&
+\frac{1}{L!}\int_{0}^{1}(1-\tau)^{L}
{\rm ad}_{\chi^{(1)}}^{L+1}[Z_{2}
+\sum_{j=n}^{r-1}H_{j}]\circ\Phi^{\tau}_{\chi^{(1)}}\mathrm{d}\tau\label{eq:35}
\\&
+R_{r}\circ\Phi_{\chi^{(1)}}\,.\label{eq:36}
\end{align}
We study each summand separately. First of all, 
by definition of $\chi_{k}^{(1)}$ (see \eqref{homoeq1} in Lemma \ref{lem:homoeq}),
we deduce that
\begin{equation}\label{emiliapara}
\sum_{k=n}^{2n-3}\big(H_{k}+\{\chi_{k}^{(1)},Z_2\}\big)
=\sum_{k=n}^{2n-3}Z_{k}+\widetilde{K}^{>N}\,,
\quad 
\widetilde{K}^{>N}:=\sum_{k=n}^{2n-3}H_{k}^{>N}.
\end{equation} 
One can check, using Lemma \ref{lem:hamvec} (see \eqref{est:smooth-vec-fie}), 
that $\widetilde{K}^{>N}$ satisfies  \eqref{KmagNN22}.
Consider now the term in \eqref{eq:34}. 
By definition of $\chi^{(1)}$ 
(see  \eqref{homoeq1} and \eqref{hamchi111}),  
we get, for $p=2,\ldots, L$,
\[
{\rm ad}^{p}_{\chi^{(1)}}[Z_2]={\rm ad}^{p-1}_{\chi^{(1)}}\big[
\{\chi^{(1)},Z_{2}\}
\big]\stackrel{\eqref{emiliapara}}{=}
{\rm ad}^{p-1}_{\chi^{(1)}}\big[ \sum_{k=n}^{2n-3}(Z_{k}-H_{k}^{\leq N})\big]\,.
\]
Therefore, by Lemma \ref{lem:hamvec}-$(ii)$ and recalling \eqref{adjointiter}, we get
\[
\eqref{eq:34}=\sum_{k=2n-2}^{L(2n-3)+r-1-2L}\widetilde{K}_{k}
\]
where $\widetilde{K}_k$ are $k$-homogeneous Hamiltonians 
in $\mathcal{L}_{k}$.
In particular, by \eqref{stimaPoisbra} and \eqref{stimacoeffchik} (with $\delta\rightsquigarrow\delta'$), we have
\[
|(\widetilde{K}_k)_{\s,j}|\lesssim_{\delta'} 
N^{L\delta'} \mu_3(j)^{\beta}\mu_{1}(j)^{-q_{d}}
\]
for some $\beta>0$ depending only on $d,n$.
This implies the estimates 
\eqref{coeffTildeK} taking $L\delta'\leq \delta$, where 
$L$ will be fixed later.
%as long as $L\leq 15$.
%of the form \eqref{K6} with coefficients $(\widetilde{K}_k)_{\s,j}$ satisfying 
%\eqref{coeffTildeK}.
Then formula \eqref{newHam1} follows by setting
\begin{equation}\label{calRR}
\mathcal{R}_{r}:=\sum_{k=r}^{L(2n-3)+r-1-2L}\widetilde{K}_{k}+
\eqref{eq:35}+\eqref{eq:36}\,.
\end{equation}
%We have to prove that \eqref{stimacalRR} holds.
The estimate \eqref{stimacalRR} holds true for $X_{\widetilde{K}_k}$ with 
$k=r,\ldots,L(2n-3)+r-1-2L$, thanks to \eqref{coeffTildeK} and 
Lemma \ref{lem:hamvec}. It remains to study the terms appearing in 
\eqref{eq:35}, \eqref{eq:36}. 
We start with the remainder in \eqref{eq:36}. We note that
\[
X_{R_{r}\circ\Phi}(u)=
({\rm d}\Phi_{\chi^{(1)}})^{-1}(u)\big[X_{R_r}(\Phi_{\chi^{(1)}}(u))\big]\,.
\]
%By \eqref{stimacoeffchik} we have that the generator 
%$\chi^{(1)}$ satisfies the hypothesis of Lemma \ref{lem:hamvec}. 
%Therefore Lemma \ref{lem:hamflow} applies
%and we 
We obtain the estimate \eqref{stimacalRR} on the
 vector field $X_{R_{r}\circ\Phi}$ by using \eqref{R8} and Lemma \ref{lem:flowmap}.
 In order to estimate the term in \eqref{eq:35} 
 we reason as follows. First notice that
 \[
 {\rm ad}_{\chi^{(1)}}^{L+1}[Z_{2}+H_{j}]
 \stackrel{\eqref{emiliapara}}{=}
 {\rm ad}_{\chi^{(1)}}^{L}
 \big[\sum_{k=n}^{2n-3}(Z_{k}-H_{k}^{\leq N})\big]+
 {\rm ad}_{\chi^{(1)}}^{L+1}[H_{j}]:=\mathcal{Q}_{j}
 \]
 with $j=n,\ldots,r-1$.
 Using Lemma
 \ref{lem:hamvec} we deduce that
 \[
 \|X_{\mathcal{Q}_{j}}(u)\|_{H^{s}}
 \lesssim_{\delta'} N^{(L+1)\delta'} \|u\|_{H^{s}}^{(Ln+n-2L)-1}\,.
 \]
 We choose $L=9$ which implies   $Ln+n-2L\geq r$ since $r\leq 4n$.
 Notice also that all the summand in \eqref{eq:35} are of the form
 \[
 \int_{0}^{1}(1-\tau)^{L}\mathcal{Q}_{j}\circ\Phi^{\tau}_{\chi^{(1)}}\mathrm{d}\tau\,.
 \]
 Then we can estimates their vector 
 fields by reasoning as done for the Hamiltonian
 $R_{r}\circ\Phi_{\chi^{(1)}}$. This concludes the proof.
 \end{proof}

\begin{remark}{\bf (Case $d=3$).}
We remark that Theorem \ref{BNFstep} for $d=3$
follows by Lemmata 
\ref{lem:homoeq}, \ref{lem:flowmap}, \ref{coniugo1}, 
by setting
$\tau^{(1)}:=\Phi_{\chi^{(1)}}$ and recalling 
that (see \eqref{def:Mnd}) $M_{d,n}=2n-2$ for $d=3$.
\end{remark}

\vspace{0.5em}
\noindent
{\bf Step 2 if $d=2$.} This step is performed only in the case $d=2$.
Consider the Hamiltonian in \eqref{newHam1}.
Our aim is to reduce in Birkhoff normal form 
all the Hamiltonians $\widetilde{K}_{k}$
of homogeneity $k=2n-2\,\ldots, M_{2,n}-1$ where 
$M_{2,n}$ is given in \eqref{def:Mnd}. 
We follow the same strategy adopted in the previous step.

\begin{lemma}{\bf (Homological equation 2).}\label{lem:homoeq2}
Let $N\geq1$, $\delta>0$ %as in Theorem \ref{BNFstep} 
and 
consider the Hamiltonian in \eqref{newHam1}.
There exist multilinear Hamiltonians 
$\chi^{(2)}_{k}$, $k=2n-2,\ldots , M_{2,n}-1$ 
in the class $\mathcal{L}_{k}$, with coefficients satisfying
%\begin{equation}\label{def:chichichi}
%\chi^{(2)}_{k}(u)= 
%\sum_{\substack{\sigma\in\{-1,1\}^k,\ j\in(\Z^{d})^k\\
%\sum_{i=1}^k\sigma_i j_i=0, \mu_2(j)\leq N}}
%(\chi^{(2)}_{k})_{\sigma,j}u_{j_1}^{\sigma_1}\cdots u_{j_k}^{\sigma_k}\,,
%\quad
%(\chi^{(2)}_{k})_{\sigma,j}\in \mathbb{C}
%\end{equation}
%with coefficients satisfying 
\begin{equation}\label{stimacoeffchikstep2}
|(\chi_{k}^{(2)})_{\s,j}|\lesssim_{\delta}
N^{\delta} \mu_3(j)^{\beta}\,,
\end{equation}
for some $\beta>0$,
such that  %(recall Def. \ref{def:splitting})
\begin{equation}\label{homoeq2}
\{\chi_{k}^{(2)},Z_{2}\}+\widetilde{K}_{k}=Z_{k}+\widetilde{K}_{k}^{>N}\,,
\quad k=2n-2,\ldots,M_{2,n}-1\,,
\end{equation}
where $\widetilde{K}_k$ are given in Lemma \ref{coniugo1}
and $Z_{k}$ is the resonant Hamiltonian defined as
\begin{equation}\label{resonantHamZZstep2}
Z_{k}:=(\widetilde{K}_{k}^{\leq N})^{\rm res}\,,\quad k=2n-2,\ldots,M_{2,n}-1\,.
\end{equation}
Moreover $Z_{k}$ belongs to $\mathcal{L}_{k}$ and has coefficients satisfying \eqref{Z4}.
\end{lemma}

\begin{proof}
Recalling Definitions \ref{Ham:class}, \ref{def:adZ2}, we write
\begin{equation*}%\label{resonantHamZZstep2}
\widetilde{K}_{k}=Z_{k}+\big(\widetilde{K}_{k}^{\leq N} -Z_{k}\big)
+\widetilde{K}_{k}^{> N}\,,
\end{equation*}
with $Z_{k}$ as in \eqref{resonantHamZZstep2}, 
and we define 
\begin{equation}\label{resonantHamZZ2step2}
\chi_{k}^{(2)}:=({\rm ad}_{Z_2})^{-1}\big[\widetilde{K}_{k}^{\leq N}-Z_{k}\big]\,,
\quad k=2n-2,\ldots,M_{2,n}-1\,.
\end{equation}
The Hamiltonians $\chi_{k}^{(2)}$
have the form \eqref{HamG}
with coefficients 
\begin{equation}\label{coeffchikstep2}
(\chi_{k}^{(2)})_{\s,j}:=(\widetilde{K}_{k})_{\s,j}
\Big(\ii \sum_{i=1}^{k}\s_i\omega_{j_i}\Big)^{-1}
\end{equation}
for indices $\s\in\{-1,+1\}^{k}$, $j\in(\Z^{d})^{k}$ such that
\[
\sum_{i=1}^{k}\s_i j_i=0\,, \;\;\;\mu_{2}(j)\leq N\;\;\; 
{\rm and}\;\;\;
\sum_{i=1}^{k}\s_i\omega_{j_i}\neq 0\,.
\]
Recalling that we are in the case $d=2$, 
by \eqref{coeffTildeK} and Proposition \ref{mainNR}  we deduce
%\begin{equation}\label{stimacoeffchikstep2}
%|(\chi_{k}^{(2)})_{\s,j}|\lesssim
%N^{2\delta} \mu_3(j)^{\beta}
%\end{equation}
\eqref{stimacoeffchikstep2}.
The resonant Hamiltonians $Z_{k}$ in \eqref{resonantHamZZstep2}
have the form \eqref{Z4}. The \eqref{homoeq2} follows by an explicit computation.
\end{proof}

\begin{lemma}\label{lem:flowmap2}
Let us define
\begin{equation}\label{hamchi112}
\chi^{(2)}:=\sum_{k=2n-2}^{M_{2,n}-1}\chi^{(2)}_{k}\,.
\end{equation}
There is $s_0=s_0(d,r)$ such that for any $\delta>0$, for any $N\geq1$ and 
any $s\geq s_0$, if
$\eps_0\lesssim_{s,\delta} N^{-\delta}$, 
then
the problem
\[
\left\{\begin{aligned}
&\pa_{\tau}Z(\tau)=X_{{\chi^{(2)}}}(Z(\tau))\\
&Z(0)=U=\vect{u}{\bar{u}}\,,\quad u\in B_{s}(0,\eps_0)
\end{aligned}\right.
\]
has a unique solution $Z(\tau)=\Phi^{\tau}_{\chi^{(2)}}(u)$
belonging to $C^{k}([-1,1];H^{s}(\mathbb{T}^{d}))$ 
for any $k\in\mathbb{N}$. Moreover the map 
$\Phi_{\chi^{(2)}}^{\tau} : B_{s}(0,\eps_0)\to H^{s}(\mathbb{T}^{d})$
is symplectic.
% w.r.t. the symplectic form \eqref{dueforma}.
The flow map $\Phi^{\tau}_{\chi^{(2)}}$, 
and its inverse 
$\Phi^{-\tau}_{\chi^{(2)}}$, satisfy
\begin{equation*}%\label{stimaflussoStep1}
\begin{aligned}
&\sup_{\tau\in[0,1]}\|\Phi^{\pm\tau}_{\chi^{(2)}}(u)-u\|_{H^{s}}
\lesssim_{s,\delta}  N^{\delta}\|u\|_{H^{s}}^{n-1}\,,\\
&\sup_{\tau\in[0,1]}
\|d\Phi^{\pm\tau}_{\chi^{(2)}}(u)[\cdot]\|_{\mathcal{L}(H^{s};H^{s})}
\leq 2\,.
\end{aligned}
\end{equation*}
\end{lemma}
\begin{proof}
It follows reasoning as in the proof of Lemma \ref{lem:flowmap}.
%Thanks to the bound \eqref{stimacoeffchikstep2} 
%and Lemmata \ref{lem:hamvec}, 
%\ref{lem:hamflow} we have that 
%the flow map is well-posed. 
%It is symplectic since it is generated by a Hamiltonian vector field.
\end{proof}

We have the following.

\begin{lemma}{\bf (The new Hamiltonian 2).}\label{coniugo2}
There is $s_0=s_0(d,r)$ such that for any $N\geq1$, $\delta>0$
and any $s\geq s_0$, if 
$\eps_0\lesssim_{s,\delta} N^{-\delta}$ then
we have that
%
%Consider the map 
%$\Phi_{\chi^{(2)}}:=(\Phi^{\tau}_{\chi^{(2)}})_{|\tau=1}$.
%There is $s_0=s_0(d,r)$ and 
%$\eps_0\lesssim N^{-16\delta}$ such that for $s\geq s_0$
%the following holds. We have that
$H\circ\Phi_{\chi^{(1)}}\circ\Phi_{\chi^{(2)}}$
has the form \eqref{HamTau1}
and satisfies items $(i),(ii),(iii)$ of Theorem \ref{BNFstep}.
\end{lemma}

\begin{proof}
We fix $\delta>0$ and we apply Lemmata \ref{coniugo1}, 
\ref{lem:homoeq2} with $\delta\rightsquigarrow \delta'$ with $\delta'$
to be chosen small enough with respect to $\delta$ fixed here.
 %\ref{lem:flowmap2}
%We reason as in Lemma \ref{coniugo1}.
Reasoning as in the previous step we have 
(recall \eqref{def:Mnd}, \eqref{adjointiter} and \eqref{newHam1})
\begin{align}
H\circ&\,\Phi_{\chi^{(1)}}\circ\Phi_{\chi^{(2)}}=
Z_{2}+\sum_{k=n}^{2n-3}Z_{k}
+\sum_{k=2n-2}^{M_{2,d}-1} \Big(\widetilde{K}_{k}+
\{\chi_{k}^{(2)},Z_2\}\Big)+\sum_{k=M_{2,n}}^{r-1}\widetilde{K}_{k}\label{eq:488}
\\&
+\widetilde{K}^{>N}\circ\Phi_{\chi^{(2)}}\label{eq:4888}
\\&
+\sum_{p=2}^{L}\frac{1}{p!}{\rm ad}_{\chi^{(2)}}^{p}[Z_2]+
\sum_{p=1}^{L}\frac{1}{p!}{\rm ad}_{\chi^{(2)}}^{p}\big[
\sum_{k=n}^{2n-3}Z_{k}+\sum_{k=2n-2}^{r-1}\widetilde{K}_{k}
\big]\label{eq:indefinit}
\\&
+%\widetilde{K}^{>N}\circ\Phi_{\chi^{(2)}}+
\mathcal{R}_{r}\circ\Phi_{\chi^{(2)}}
+\frac{1}{L!}\int_{0}^{1}(1-\tau)^{L}
{\rm ad}^{L+1}_{\chi^{(2)}}[Z_2]\circ\Phi_{\chi^{(2)}}^{\tau}d\tau\label{eq:47}
\\&
+\frac{1}{L!}\int_{0}^{1}(1-\tau)^{L}{\rm ad}^{L+1}_{\chi^{(2)}}
\Big[\sum_{k=n}^{2n-3}Z_{k}
+\sum_{k=2n-2}^{r-1}\widetilde{K}_{k}\Big]\circ\Phi_{\chi^{(2)}}^{\tau}\mathrm{d}\tau\label{eq:48}\,,
\end{align}
where
$\Phi^{\tau}_{\chi^{(2)}}$, $\tau\in[0,1]$, is the flow at time $\tau$ of the Hamiltonian 
$\chi^{(2)}$.
We study each summand separately. First of all, thanks to  
\eqref{homoeq2},
%\eqref{resonantHamZZstep2}, \eqref{resonantHamZZ2step2},
we deduce that
\begin{equation}\label{emiliapara2}
\sum_{k=2n-2}^{M_{2,n}-1}\big(\widetilde{K}_{k}+\{\chi_{k}^{(2)},Z_2\}\big)
=\sum_{k=2n-2}^{M_{2,n}-1}Z_{k}+\widetilde{K}_{+}^{>N}\,,
\quad
\widetilde{K}_{+}^{>N}:=\sum_{k=2n-2}^{M_{2,n}-1}\widetilde{K}_{k}^{>N}\,.
\end{equation}
One can check, using Lemma \ref{lem:hamvec}, 
that $\widetilde{K}_{+}^{>N}$ satisfies
\begin{equation}\label{parmigiano}
\|X_{\widetilde{K}_{+}^{>N}}\|_{H^{s}}\lesssim_{s,\delta} N^{-1+\delta'}\|u\|_{H^{s}}^{2n-3}\,.
\end{equation}
Consider now the terms in \eqref{eq:indefinit}.
%We now analyze the first summand  in \eqref{eq:48}.
First of all notice that we have
\[
{\rm ad}_{\chi^{(2)}}^{p}[Z_2]\stackrel{\eqref{emiliapara2}}{=}
\sum_{k=2n-2}^{M_{2,n}-1}{\rm ad}^{p-1}_{\chi^{(2)}}\big[
Z_{k}-\widetilde{K}_{k}^{\leq N}
\big]\,, \;\;\;p=2,\ldots,L\,.
\]
The Hamiltonian above has a homogeneity at least of degree
$4n-6$ which actually is larger or equal to $M_{2,n}$ (see \eqref{def:Mnd}).
The terms with lowest homogeneity in the sum \eqref{eq:indefinit}
have degree exactly $M_{2,n}$ and come from the 
term ${\rm ad}_{\chi^{(2)}}\big[\sum_{k=n}^{2n-3}Z_{k}\big]$
recalling that (see Remark \ref{Rmk:resonant}) if $n$ is odd then 
$Z_{n}\equiv0$.
Then, by \eqref{stimacoeffchikstep2}, \eqref{coeffTildeK} and 
Lemma \ref{lem:hamvec}-$(ii)$, we get
\[
\eqref{eq:indefinit}=\sum_{k=M_{2,n}}^{L(M_{2,n}-1)+r-1-2L}\widetilde{K}^{+}_{k}
\]
where $\widetilde{K}^{+}_k$ are $k$-homogeneous Hamiltonians 
of the form \eqref{K6} with coefficients %$(\widetilde{K}^{+}_k)_{\s,j}$ 
satisfying 
\begin{equation}\label{coeffTildeKstep2}
|(\widetilde{K}^{+}_k)_{\s,j}|\lesssim_{\delta'} N^{(L+1)\delta'} \mu_3(j)^{\beta}\,,
\end{equation}
for some $\beta>0$.
By the discussion above, using formul\ae \, \eqref{eq:488}-\eqref{eq:48},
% and setting
%\begin{equation}\label{def:tau1}
%\tau^{(1)}:=\Phi_{\chi^{(1)}}\circ\Phi_{\chi^{2}}\,,
%\end{equation}
we obtain that  the Hamiltonian 
$H\circ\Phi_{\chi^{(1)}}\circ\Phi_{\chi^{(2)}}$
has the form \eqref{HamTau1}
with (recall \eqref{resonantHamZZ}, \eqref{resonantHamZZstep2}, 
\eqref{eq:4888}, \eqref{emiliapara2})
\begin{align}
Z_{k}^{\leq N}&:=Z_{k}\,,\qquad \;\;
K_{k}:=\widetilde{K}_{k}+\widetilde{K}_{k}^{+}\,,\qquad k=M_{2,n},\ldots,r-1\,,
\label{finalresonant}\\
K^{>N}&:=\widetilde{K}^{>N}\circ\Phi_{\chi^{(2)}}+ \widetilde{K}_{+}^{>N}
\label{finalsmooth}
\end{align}
and with remainder $\tilde{R}_{r}$ defined as
\begin{equation}\label{calRRstep2}
\tilde{R}_{r}:=\sum_{k=r}^{L(M_{2,n}-1)+r-1-2L}\widetilde{K}_{k}^{+}+
\eqref{eq:47}+\eqref{eq:48}\,.
\end{equation}
%
%\vspace{3em}
%Moreover, using \eqref{stimacoeffchikstep2}, \eqref{coeffTildeK} and
%Lemmata \ref{lem:poiss}, \ref{lem:hamvec}, we have
%\[
%\|X_{{\rm ad}_{\chi^{(2)}}^{2}[Z_2]}\|_{H^{s}}\lesssim N^{2\delta}\|u\|_{H^{s}}^{4n-7}\,.
%\]
Recalling \eqref{resonantHamZZ}, \eqref{resonantHamZZstep2}
and the estimates \eqref{initCoeff},  \eqref{coeffTildeK}
we have that $Z_{k}^{\leq N}$ in \eqref{finalresonant}
satisfies the condition of item $(i)$ of Theorem \ref{BNFstep}.
Similarly $K_{k}$ in \eqref{finalresonant}  satisfies 
\eqref{K6} thanks to \eqref{coeffTildeK} and 
\eqref{coeffTildeKstep2} as long as $\delta'$ is sufficiently small.
The remainder $K^{>N}$ in \eqref{finalsmooth} satisfies the
bound \eqref{KmagNN} using \eqref{parmigiano},  \eqref{KmagNN22}
and Lemma \ref{lem:hamvec}-$(i)$. 
It remains to show that the remainder defined in \eqref{calRRstep2}
satisfies the estimate \eqref{R8tilde}.
The claim follows for the terms $\widetilde{K}_{k}^{+}$
for $k=r,\ldots, L(M_{2,n}-1)+r-1-2L$
by using \eqref{coeffTildeKstep2} and Lemma \ref{lem:hamvec}. 
For the remainder in \eqref{eq:47}, \eqref{eq:48} one can reason 
following almost word by word 
the proof of the estimate of the vector field of $\mathcal{R}_r$
in \eqref{calRR} in the previous step. In this case we choose $L+1=8$ which implies
$L+1\geq (r+n)/(2n-4)$.
\end{proof}

Theorem \ref{BNFstep} follows by Lemmata \ref{coniugo1}, \ref{coniugo2}
setting $\tau^{(1)}:=\Phi_{\chi^{(1)}}\circ\Phi_{\chi^{(2)}}$. The bound 
\eqref{tau} follows by Lemmata \ref{lem:flowmap} and \ref{lem:flowmap2}.

\section{The modified energy step}\label{section:modif}

In this section we construct a modified energy 
which is an approximate constant of motion 
for the Hamiltonian system of 
$H\circ\tau^{(1)}$ in \eqref{HamTau1}, when $d=2,3$, and for the Hamiltonian $H$
in \eqref{expinitial} when  $d\geq4$.
For compactness we shall write, for $s\in \mathbb{R}$,
\begin{equation}\label{HamNN}
N_{s}(u):=\|u\|_{H^{s}}^{2}=\sum_{j\in\mathbb{Z}^{2}}\langle j\rangle^{2s}|u_{j}|^{2}\,,
\end{equation}
for $u\in H^{s}(\mathbb{T}^{2};\mathbb{C})$.
For $d\geq2$ and $n\in \mathbb{N}$ we define (recall \eqref{def:Mnd})
\begin{equation}\label{def:Mndtilde}
\widetilde{M}_{d,n}:=\left\{
\begin{aligned}
&M_{d,n}+n-1 \quad  n\;\; {\rm odd}\\
&M_{d,n}+n-2 \quad  n\;\; {\rm even}.
\end{aligned}\right.
\end{equation}
%%%%THIS IS ANOTHER WAY TO WRITE MTILDE%%%%%%
%\begin{equation}\label{def:Mndtilde}
%\widetilde{M}_{d,n}:=\left\{
%\begin{aligned}
%&4n-4\quad {\rm if}\;\; d=2\;\; {\rm and}\;\; n\;\; {\rm odd}\\
%&4n-6\quad {\rm if}\;\; d=2\;\; {\rm and}\;\; n\;\; {\rm even}\\
%&3n-3\quad {\rm if}\;\; d=3\;\; {\rm and}\;\; n\;\; {\rm odd}\\
%&3n-4\quad {\rm if}\;\; d=3\;\; {\rm and}\;\; n\;\; {\rm even}\\
%&2n-2\quad {\rm if}\;\; d\geq4
%\end{aligned}\right.
%\end{equation}

\begin{proposition}\label{modifiedstep}
%Let $\delta>0$, $N\geq1$ as in Theorem \ref{BNFstep} and consider also
%$N\geq N_1\geq1$.
%For $d\geq4$ consider $N=N_1$.
 There exists $\beta=\beta(d,n)>0$ such that for any $\delta>0$, any  $N\geq N_1>1$
 %($N=N_1$ if $d\geq4$)
and any $s\geq\tilde{s}_0$, 
for some $\tilde{s}_0=\tilde{s}_0(\beta)>0$, if 
$\eps_0 \lesssim_{s,\delta} N^{-\delta}$, there are  multilinear maps $E_{k}$, 
$k=M_{d,n},\ldots, \widetilde{M}_{d,n}-1$, 
in the class $\mathcal{L}_{k}$
%of the form
%\begin{equation}\label{energiesE6E7}
%E_{k}= 
%\sum_{\substack{\sigma\in\{-1,1\}^k,\ j\in(\Z^{d})^k\\
%\sum_{i=1}^{k}\sigma_i j_i=0}}
%(E_{k})_{\sigma,j}u_{j_1}^{\sigma_i}\cdots u_{j_d}^{\sigma_d}\,,
%\end{equation}
%with $(E_{k})_{\sigma,j}\in \mathbb{C}$ 
such that the following holds:

\noindent
$\bullet$ the coefficients 
$(E_{k})_{\sigma,j}$ satisfies
\begin{equation}\label{coeffE6}
|(E_{k})_{\sigma,j}|\lesssim_{s,\delta}N^{\delta} N_1^{\kappa_{d}}
\mu_{3}(j)^{\beta} \mu_1(j)^{2s}\,,
\end{equation}
for $\sigma\in \{-1,1\}^{k}$, $j\in(\mathbb{Z}^{d})^{k}$, 
$k=M_{d,n},\ldots,\widetilde{M}_{d,n}-1$, where
\begin{equation}\label{def:mud}
\kappa_d:=0\; {\rm if}\; d=2\,,\;\;\;\;\;\; 
\kappa_d:=1\; {\rm if}\; d=3\,,\;\;\;\;\;\; 
\kappa_d:=d-4\; {\rm if}\; d\geq4\,.
\end{equation}

\noindent
$\bullet$ for any $u\in B_{s}(0,2\eps_0)$ setting 
\begin{equation}\label{EEE}
E(u):=\sum_{k=M_{d,n} }^{\widetilde{M}_{d,n}-1}E_{k}(u)\,.
\end{equation}
one has
\begin{equation}\label{energyestimate}
\begin{aligned}
|\{N_{s}+E,H\circ\tau^{(1)}\}|&\lesssim_{s,\delta} 
N_1^{\kappa_d} N^{\delta}
\big(\|u\|^{\widetilde{M}_{d,n}}_{H^{s}}
+N^{-1}\|u\|^{M_{d,n}+n-2}_{H^{s}}\big)\\
&+N_1^{-\mathfrak{s}_d+\delta}\|u\|_{H^{s}}^{M_{d,n}}
+N^{-\mathfrak{s}_d+\delta}\|u\|_{H^{s}}^{n}\,,
\end{aligned}
\end{equation}
where 
\begin{equation}\label{def:fraks}
\mathfrak{s}_{d}:=1\,,\;\;\;{\rm for }\;\; d=2,3\,,\;\;\; {\rm and}\;\;\;
\mathfrak{s}_{d}:=3\,,\;\;\;{\rm for }\;\; d\geq4\,.
\end{equation}
\end{proposition}

We need the following technical lemma.

\begin{lemma}{\bf (Energy estimate).}\label{lem:energia}
Let $N\geq1$, $0\leq \delta<1$, $p\in \mathbb{N}$, $p\geq3$. 
Consider the Hamiltonians $N_{s}$ in \eqref{HamNN}, 
$G_p\in \mathcal{L}_{p}$
and write $G_{p}=G_{p}^{(+1)}+G_{p}^{(-1)}$ (recall Definition \ref{def:adZ2}). 
Assume also that  the coefficients
of $G_p$ satisfy
\begin{equation}\label{coeffGp222}
|(G^{(\eta)}_{p})_{\sigma,j}|\leq C
 N^{\delta}\mu_{3}(j)^{\beta}\mu_{1}(j)^{-q}\,,
 \qquad \forall \s\in\{-1,+1\}^{p}\,,\; j\in \mathbb{Z}^{d}\,,\eta\in\{-1,+1\},
\end{equation}
for some $\beta>0$, $C>0$ and $q\geq0$.  
We have that the Hamiltonian $Q_{p}^{(\eta)}:=\{N_s,G_{p}^{(\eta)}\}$, $\eta\in\{-1,1\}$,
belongs to the class $\mathcal{L}_{p}$ 
and has coefficients satisfying 
\begin{equation}\label{coeffQp}
|(Q_{p}^{(\eta)})_{\sigma,j}|\lesssim_{s} 
C N^{\delta}\mu_{3}(j)^{\beta+2}\mu_1(j)^{2s}
\mu_1(j)^{-q-\alpha}\,,\qquad 
\alpha:=\left\{
\begin{aligned}
& 1\,,\;\;\; {\rm if }\;\; \eta=-1\\
&0\,, \;\;\; {\rm if}\;\; \eta=+1\,.
\end{aligned}
\right.
%\mu_1(j)^{-q-(1-\eta )/2},\quad 
%\eta\in\{-1,+1\}\,.
\end{equation}
\end{lemma}

\begin{proof}
Using formul\ae \, \eqref{HamNN},  \eqref{Poissonbrackets}, \eqref{HamG}
and recalling Def. \ref{def:adZ2}
we have that the Hamiltonian $\{N_s,G_{p}^{(\eta)}\}$ has coefficients
\[
(Q_{p}^{(\eta)})_{\sigma,j}=(G_{p}^{(\eta)})_{\sigma,j}
\ii \Big(\sum_{i=1}^{p}\s_i\langle j_i\rangle^{2s}\Big)
\]
for any $\s\in\{-1,+1\}^p$, $j\in (\mathbb{Z}^{d})^{p}$ satisfying 
\[
\sum_{i=1}^p\sigma_i j_i=0\,,\quad  \sigma_i \sigma_k=\eta\,,\quad
\mu_1(j)=|j_i|\,, \;\mu_2(j)=|j_k|\,,
\]
for some $i,k=1,\ldots,p$.
Then the bound \eqref{coeffQp} follows by the fact that
\[
|\langle j_i\rangle^{2s}+\eta\langle j_k\rangle^{2s}|\lesssim_{s}
\left\{\begin{aligned}
&\mu_1(j)^{2s-1}\mu_{3}(j)\;\;\;{\rm if} \; \eta=-1\\
&\mu_1(j)^{2s}\qquad\;\;\;\;\;\;\;{\rm if} \; \eta=+1\,.
\end{aligned}\right.
\]
and using the assumption \eqref{coeffGp222}.
\end{proof}

\begin{proof}[{\bf  Proof of Proposition \ref{modifiedstep}}]
\emph{Case $d=2,3$.}
Consider the Hamiltonians $K_{k}$ in \eqref{K6} for 
$k=M_{d,n},\ldots, \widetilde{M}_{n,d}-1$ 
where
$ \widetilde{M}_{n,d}$  is defined in \eqref{def:Mndtilde}.
Recalling Definition \ref{def:adZ2} 
we set
$E_{k}:=E_{k}^{(+1)}+E_{k}^{(-1)}$,
where 
\begin{equation}\label{def:EEk}
E_{k}^{(+1)}:=({\rm ad}_{Z_2})^{-1}\{N_{s}, K_{k}^{(+1)}\}\,,
\qquad E_{k}^{(-1)}:=({\rm ad}_{Z_2})^{-1} \{N_{s}, K_{k}^{(-1,\leq N_1)}\}\,,
\end{equation}
for $k=M_{d,n},\ldots, \widetilde{M}_{d,n}-1$. 
It is easy to note that   $E_{k}\in \mathcal{L}_{k}$.
% have the form
%\eqref{energiesE6E7}.
Moreover, using the bounds on the coefficients $(K_{k})_{\s,j}$ in \eqref{K6}
and 
Proposition \ref{mainNR}
(with $\delta$ therein possibly smaller than the one fixed here), 
one can check that the coefficients 
$(E_{k})_{\s,j}$ satisfy the \eqref{coeffE6}.
Using  \eqref{def:EEk} we notice that
\begin{equation}\label{omoomoeq}
\{N_{s}, K_{k}\}+\{E_{k},Z_2\}=\{N_{s}, K_{k}^{(-1,>N_1)}\}\,,\quad 
k=M_{d,n},\ldots, \widetilde{M}_{d,n}-1\,.
\end{equation}
Combining Lemmata \ref{lem:hamvec} and 
\ref{lem:energia} we deduce 
\begin{equation}\label{stimasmooth}
|{\{N_{s}, K_{k}^{(-1,>N_1)}\}}(u)|\lesssim_{s,\delta}
N_1^{-1+\delta}\|u\|^{k}_{H^{s}}\,,
\end{equation}
for $s$ large enough with respect to $\beta$.
We define the energy $E$ as in \eqref{EEE}.
We are now in position to prove the estimate \eqref{energyestimate}.

\noindent Using the expansions  \eqref{HamTau1} and \eqref{EEE} we get
\begin{align}
\{N_{s}+E,H\circ\tau^{(1)}\}&=
\{N_{s},Z_2+ 
\sum_{k=n}^{M_{d,n}-1}Z_{k}^{\leq N}\}\label{energia1}\\
&+\{N_{s},K^{>N}\}+\{N_{s},\tilde{R}_{r}\}\label{energia2}\\
&+\sum_{k=M_{d,n}}^{\widetilde{M}_{d,n}-1}\Big( \{N_{s}, K_{k}\}+\{E_{k},Z_2\}\Big)
\label{energia3}\\
&+\{E,\sum_{k=n}^{M_{d,n}-1}Z_{k}^{\leq N}\}
+\{E,\sum_{k=M_{d,n}}^{r-1}K_{k}+\tilde{R}_r\}\label{energia4}\\
&+\{E, K^{>N}\}.
\label{energia5}
%\\&+\{E,\sum_{k=M_{d,n}}^{r-1}K_{k}+\tilde{R}_r\}\,,\label{energia6}
\end{align}
We study each summand separately.
First of all note that, by item $(i)$ in Theorem \ref{BNFstep} and Proposition
\ref{mainNR} we deduce 
that the right hand side of \eqref{energia1}  vanishes.
Consider now the term in \eqref{energia2}. Using the bounds
\eqref{KmagNN}, \eqref{R8tilde} and recalling \eqref{Poissonbrackets} 
one can check that, for $\eps_0N^{\delta}\lesssim_{s,\delta}1$,
\begin{equation}\label{stima75}
|\eqref{energia2}|\lesssim_{s,\delta} N^{-1+\delta}\|u\|_{H^{s}}^{n}+
N^{\delta}\|u\|_{H^{s}}^{r}\,.
\end{equation}
By \eqref{omoomoeq} and \eqref{stimasmooth} 
we deduce that
\begin{equation}\label{stima76}
|\eqref{energia3}|\lesssim_{s,\delta}N_1^{-1+\delta}\|u\|^{M_{d,n}}_{H^{s}}\,.
\end{equation}
By \eqref{coeffE6}, \eqref{HamTau1}-\eqref{R8tilde}, Lemma % \ref{lem:poiss} and 
\ref{lem:hamvec} (recall also \eqref{def:Mndtilde}) we get
\[
\begin{aligned}
|\eqref{energia4}|%+|\eqref{energia6}|
&\lesssim_{s,\delta} 
N_1^{\kappa_{d}}N^{\delta}(\|u\|_{H^{s}}^{\widetilde{M}_{d,n}}+\|u\|_{H^{s}}^{r})\,,\\
|\eqref{energia5}|&\lesssim_{s,\delta}
N_1^{\kappa_{d}}N^{-1+\delta}\|u\|_{H^{s}}^{{M}_{d,n}+n-2}\,.
\end{aligned}
\]
The discussion above implies the bound \eqref{energyestimate}
using that $r\geq \widetilde{M}_{d,n}$. 
This concludes the proof in the case $d=2,3$.

\vspace{0.2em}
\noindent
\emph{Case $d\geq4$.} In this case we consider the Hamiltonian $H$
in \eqref{expinitial}. Recalling Definition \ref{def:adZ2} %and \ref{def:splitting}
we set
\[
E_{k}:=E_{k}^{(+1)}+E_{k}^{(-1)}
\]
where 
\begin{equation}\label{def:EEkd4}
E_{k}^{(+1)}:=({\rm ad}_{Z_2})^{-1}\{N_{s}, H_{k}\}^{(+1)}\,,
\qquad E_{k}^{(-1)}:=({\rm ad}_{Z_2})^{-1} \{N_{s}, H_{k}^{(-1,\leq N_1)}\}\,,
\end{equation}
for $k=M_{d,n},\ldots, \widetilde{M}_{d,n}-1$. 
Notice that the energies $E_{k}^{(+1)}$, $E_{k}^{(-1)}$ are in $\mathcal{L}_{k}$
with coefficients
\[
(E_{k}^{(+1)})_{\s,j}=\big(\sum_{i=1}^{k}\s_i\langle j_i\rangle^{2s}\big)
\big(\sum_{i=1}^{k}\s_i\omega_{j_i}\big)^{-1}(H_{k}^{(+1)})_{\s,j}\,,
\quad \s\in\{-1,+1\}^{k}\,,\;\; j\in(\Z^{d})^{k}\,,
\]  
and
\[
(E_{k}^{(-1)})_{\s,j}=\big(\sum_{i=1}^{k}\s_i\langle j_i\rangle^{2s}\big)
\big(\sum_{i=1}^{k}\s_i\omega_{j_i}\big)^{-1}(H_{k}^{(-1)})_{\s,j}\,,
\;\;\; \mu_{2}(j)\leq N_1\,,
\]  
with $\s\in\{-1,+1\}^{k}$, $j\in(\Z^{d})^{k}$.
Using Proposition \ref{mainNR} and reasoning as in the proof of Lemma 
\ref{lem:energia} one can check that estimate  \eqref{coeffE6} on the coefficients of $E_{k}^{(+1)}$ and $E_{k}^{(-1)}$
holds true with $\kappa_{d}$ as in \eqref{def:mud}.
Equation \eqref{def:EEkd4} implies
\begin{equation}\label{omoomoeqD4}
\{N_{s}, H_{k}\}+\{E_{k},Z_2\}=\{N_{s}, H_{k}^{(-1,>N_1)}\}\,,\quad 
k=M_{d,n},\ldots, \widetilde{M}_{d,n}-1\,.
\end{equation}
Recall that the coefficients of the Hamiltonian $H_{k}$ 
satisfy the bound \eqref{initCoeff}. Therefore, 
combining Lemmata %\ref{smoothvecfield}, 
\ref{lem:energia} 
and \ref{lem:hamvec},
we deduce 
\begin{equation}\label{stimasmoothD4}
|\{N_{s}, H_{k}^{(-1,>N_1)}\}(u)|
\lesssim_{s,\delta}N_1^{-3}\|u\|^{k}_{H^{s}}\,,
\end{equation}
for $s$ large enough with respect to $\beta$.
Recalling \eqref{expinitial} we have
\begin{align*}
\{N_{s}+E,H\}&=
\{N_{s},Z_2\}+\{N_{s},{R}_{r}\}
+\{E,\sum_{k=M_{d,n}}^{r-1}H_{k}+{R}_{r}\}
\\
&+\sum_{k=M_{d,n}}^{\widetilde{M}_{d,n}-1}\Big( \{N_{s}, K_{k}\}+\{E_{k},Z_2\}\Big)\,.
\end{align*}
One can obtain the bound \eqref{energyestimate} 
by reasoning as in the case $d=2,3$,
using \eqref{stimasmoothD4}, \eqref{R8}
and recalling that $\widetilde{M}_{d,n}=M_{d,n}+n-2$ 
(see \eqref{def:Mndtilde})
when $d\geq4$. This concludes the proof.
\end{proof}

\section{Proof of Theorem \ref{thm-main}}\label{section:boot}
In this section we show how to combine the results of Theorem \ref{BNFstep}
and Proposition \ref{modifiedstep}
in order to prove Theorem \ref{thm-main}.

\noindent
Consider $\psi_0$ and $\psi_1$ satisfying \eqref{initialstima}
and let $\psi(t,y)$, $y\in \mathbb{T}_{\nu}^{d}$, 
be the unique solution of \eqref{eq:beam}
with initial conditions $(\psi_0,\psi_1)$ defined for $t\in[0,T]$
for some $T>0$. By rescaling the space variable $y$ and
passing to the complex variable in \eqref{beam5} 
we consider the function $u(t,x)$, $x\in\mathbb{T}^{d} $ 
solving  the equation \eqref{eq:beamComp}.
We recall that \eqref{eq:beamComp} can be written in the Hamiltonian form
\begin{equation}\label{starwars}
\pa_{t}u=\ii \partial_{\bar{u}} H(u)\,,
\end{equation}
where $H$ is the Hamiltonian function in \eqref{beamHam} (see also \eqref{expinitial}).
We have that Theorem \ref{thm-main} is a consequence of
the following Lemma.

\begin{lemma}{\bf (Main bootstrap)}\label{main:boot}
There exists $s_0=s_0(n,d)$ such that for any $\delta>0$, $s\geq s_0$, there exists $\eps_0=\eps_0(\delta,s)$ such that the following holds.
Let $u(t,x)$ be a solution of \eqref{starwars} with 
$t\in [0,T)$, $T>0$ and initial condition
$u(0,x)=u_0(x)\in H^{s}(\mathbb{T}^{d})$. For any $\eps\in (0, \eps_0)$
if
\begin{equation}\label{hypBoot}
\|u_0\|_{H^{s}}\leq \eps\,,\quad \sup_{t\in[0,T)}\|u(t)\|_{H^{s}}\leq 2\eps\,,
\quad T\leq \eps^{-\mathtt{a}+\delta}\,,
\end{equation}
with $\mathtt{a}=\mathtt{a}(d,n)$ 
in \eqref{Atime},
then we have the improved bound 
$\sup_{t\in[0,T)}\|u(t)\|_{H^{s}}\leq \frac32 \eps$\,.
\end{lemma}

In order to prove Lemma \ref{main:boot} we first need a preliminary result.
\begin{lemma}{\bf (Equivalence of the energy norm)}\label{equivNorms}
Let $\delta>0$, $N\geq N_1\geq 1$.
% and $s_0$ as in Theorem \ref{BNFstep}
%and Proposition \ref{modifiedstep}.
Let $u(t,x)$ as in \eqref{hypBoot} with $s\gg1$ large enough. 
 Then, for any $0<c_0<1$,  there exists $C=C(\delta,s,d,n,c_0)>0$ such that, if
 we have the smallness condition
\begin{equation}\label{smalleps}
\eps C N^{\delta}N_1^{\kappa_d}\leq 1\,,
\end{equation}
the following holds true. Define
\begin{equation}\label{def:calE}
z:=\tau^{(0)}(u)\,,\quad u=\tau^{(1)}(z)\,,
\quad \mathcal{E}_{s}(z):=(N_{s}+E)(z)
\end{equation}
where $\tau^{(\s)}$, $\s=0,1$, are the maps given by Theorem \ref{BNFstep}
and $N_{s}$ is in \eqref{HamNN}, $E$ is given by Proposition \ref{modifiedstep}.
We have
\begin{equation}\label{equivZZ}
1/(1+c_0)\|z\|_{H^{s}}\leq \|u\|_{H^{s}}\leq (1+c_0)\|z\|_{H^{s}}\,,\quad\forall t\in[0,T]\,;
\end{equation}
\begin{equation}\label{equivEEE}
1/(1+12c_0)\mathcal{E}_{s}(z)\leq \|u\|^{2}_{H^{s}}
\leq (1+12c_0)\mathcal{E}_{s}(z)\,,\quad\forall t\in[0,T]\,.
\end{equation}
\end{lemma}

\begin{proof}
Thanks to \eqref{smalleps} we have that Theorem 
\ref{BNFstep} and Proposition \ref{modifiedstep} apply.
Consider the function $z=\tau^{(0)}(u)$. By estimate \eqref{tau}
we have
\[
\|z\|_{H^{s}}\leq \|u\|_{H^{s}}+\tilde{C}N^{\delta} \|u\|_{{H}^s}^2
\stackrel{(\ref{smalleps})}{\leq}\|u\|_{H^{s}}(1+c_0)\,,
\]
where $\tilde{C}$ is some constant depending on $s$ and $\delta$. 
The latter inequality  follows
by taking $C$ in \eqref{smalleps} large enough. 
Reasoning similarly and using the bound
\eqref{tau} on $\tau^{(1)}$ one gets the \eqref{equivZZ}.
Let us check the \eqref{equivEEE}. First notice that, by 
\eqref{coeffE6}, \eqref{EEE}
and Lemma \ref{lem:hamvec},
\begin{equation}\label{stimaE(z)}
|E(z)|\leq \tilde{C} \|z\|_{H^{s}}^{M_{d,n}}N^{\delta}N_1^{\kappa_{d}}\,,
\end{equation}
for some $\tilde{C}>0$ depending on $s$ and $\delta$.
Then, recalling \eqref{def:calE}, we get
\[
|\mathcal{E}_{s}(z)|\leq \|z\|^{2}_{H^{s}}
(1+ \tilde{C}\|z\|_{H^{s}}^{M_{d,n}-2}N^{\delta}N_1^{\kappa_{d}})
\stackrel{(\ref{equivZZ}), (\ref{smalleps})}{\leq }\|u\|_{H^{s}}^{2}(1+c_0)^{3}\,,
\]
where we used that $M_{d,n}-2\geq1$. 
This implies the first inequality in \eqref{equivEEE}.
On the other hand, using \eqref{equivZZ}, \eqref{stimaE(z)} 
and \eqref{hypBoot},  we have
\[
\begin{aligned}
\|u\|_{H^{s}}^{2}\leq (1+c_0)^{2}\mathcal{E}_{s}(z)
+(1+c_0)^{M_{d,n}+{2}}\tilde{C}
N^{\delta}N_1^{\kappa_{d}}\eps^{M_{d,n}-2}\|u\|_{H^{s}}^{2}\,.
\end{aligned}
\]
Then, since $M_{d,n}>2$ (see \eqref{def:Mnd}), 
taking $C$ in  \eqref{smalleps} large enough 
we obtain the second inequality in \eqref{equivEEE}.
\end{proof}

\begin{proof}[{\bf Proof of Lemma \ref{main:boot}}]
Assume the \eqref{hypBoot}. We study how the Sobolev norm
$\|u(t)\|_{H^{s}}$ evolves for $t\in[0,T]$ by inspecting  the equivalent  
energy norm $\mathcal{E}_{s}(z)$ defined in \eqref{def:calE}.
Notice that
\[
\pa_{t}\mathcal{E}_{s}(z)=-\{\mathcal{E}_{s}, H\circ\tau^{(1)}\}(z)\,.
\]
Therefore, for any $t\in [0,T]$, we have that
\[
\begin{aligned}
\left|\int_{0}^{T}\pa_{t}\mathcal{E}_{s}(z)\ \mathrm{d}t\right|
&\stackrel{\eqref{energyestimate},\eqref{hypBoot}}{\lesssim_{s,\delta}}
TN_1^{\kappa_d} N^{\delta}
\big(\eps^{\widetilde{M}_{d,n}}
+N^{-1}\eps^{M_{d,n}+n-2}\big)
\\&\quad
+TN_1^{-\mathfrak{s}_d+\delta}\eps^{M_{d,n}}
+TN^{-\mathfrak{s}_d+\delta}\eps^{n}\,.
\end{aligned}
\]
We now 
fix 
\[
N_1:=\eps^{-\alpha}\,,\quad N:=\eps^{-\gamma}\,,
\]
with $0<\alpha\leq \gamma$ to be chosen properly. Hence we 
have
\begin{align}
\left|\int_{0}^{T}\pa_{t}\mathcal{E}_{s}(z)\ \mathrm{d}t\right|&\lesssim_{s,\delta}
\eps^{2}T\Big(
\eps^{M_{d,n}-2+\alpha \mathfrak{s}_{d}-\delta\alpha}
+\eps^{\widetilde{M}_{n,d}-2-\alpha\kappa_{d}-\delta\gamma}
\Big)\label{moon1}\\
&+
\eps^{2}T\Big(
\eps^{n-2+\gamma \mathfrak{s}_{d}-\delta\gamma}
+\eps^{M_{n,d}+n-4+\gamma-\alpha\kappa_{d}-\delta\gamma}
\Big)\,.\label{moon2}
\end{align}
We choose $\alpha>0$ such that
\begin{equation}\label{alpha1}
M_{d,n}-2+\alpha \mathfrak{s}_{d}=\widetilde{M}_{n,d}-2-\alpha\kappa_{d}\,,
\end{equation}
i.e.
\begin{equation}\label{alpha2}
\alpha:=
\frac{\widetilde{M}_{n,d}-M_{d,n}}{\mathfrak{s}_d+\kappa_d}
\stackrel{\eqref{def:Mndtilde}, \eqref{def:fraks}, \eqref{def:mud}}{=}
\left\{\begin{aligned}
&\tfrac{n-1}{d-1}\;\;\;{\rm if} \; n \;\;{\rm odd}\\
&\tfrac{n-2}{d-1}\;\;\;{\rm if} \; n \;\;{\rm even}\,.
\end{aligned}\right.
\end{equation}
We shall choose $\gamma>0$ is such a way 
the terms in \eqref{moon2} are negligible with respect to the terms in 
\eqref{moon1}. In particular we set (recall \eqref{alpha2})
\begin{equation}\label{gamma1}
\gamma\geq \max\big\{M_{d,n}-4-n
+\frac{\widetilde{M}_{d,n}-M_{d,n}}{\mathfrak{s}_d+\kappa_d}\mathfrak{s}_d,
2-n+\widetilde{M}_{d,n}-M_{d,n}
\big\}\,.
\end{equation}
Therefore estimates \eqref{moon1}-\eqref{moon2}
become 
\[
\left|\int_{0}^{T}\pa_{t}\mathcal{E}_{s}(z)\ \mathrm{d}t\right|\lesssim_{s,\delta}
\eps^{2}T\eps^{\mathtt{a}}(
\eps^{-\delta\alpha}
+\eps^{-\delta\gamma})
\]
where $\mathtt{a}$ is defined in \eqref{Atime}
and appears thanks to definitions
 \eqref{def:Mnd}, \eqref{def:Mndtilde}, \eqref{def:mud}, \eqref{def:fraks} 
and \eqref{alpha2}.
Moreover we define
\[
\delta':=2\delta\max\{\alpha,\gamma\}\,,
\]
with $\alpha,\gamma$ given in \eqref{alpha2} and \eqref{gamma1}. 
Notice that, since
$\delta>0$ is arbitrary small, 
then $\delta'$ can be chosen arbitrary small.
Since $\eps$ can be chosen arbitrarily small with respect to $s$ and $\delta$, 
with this choices we get
\[
\left|\int_{0}^{T}\pa_{t}\mathcal{E}_{s}(z)\ \mathrm{d}t\right|
\leq\eps^{2}/4
\]
as long as $T\leq \eps^{-\mathtt{a}+\delta'}$.
Then, using the equivalence of norms  \eqref{equivEEE} and choosing 
$c_0>0$ small enough, we have
\[
\begin{aligned}
\|u(t)\|_{H^{s}}^{2}&\leq (1+12c_0)\mathcal{E}_0(z(t))
\\&
\leq (1+12c_0)\Big[\mathcal{E}_s(z(0))
+\left|\int_{0}^{T}\pa_{t}\mathcal{E}_{s}(z)\ \mathrm{d}t\right|\Big]
\\&
\leq (1+12c_0)^{2}\eps^{2}+(1+12c_0)\eps^{2}/4\leq \eps^{2}3/2\,,
\end{aligned}
\]
for times $T\leq \eps^{-\mathtt{a}+\delta'}$.
This implies the thesis.  
\end{proof}

\appendix

\def\cprime{$'$}

\end{document}